\documentclass[reqno]{amsart}
\usepackage[utf8]{inputenc}
\usepackage{amsmath}
\usepackage{amssymb}
\usepackage{amsthm}
\usepackage[all]{xy}
\xyoption{tips}
\SelectTips{cm}{10}
\usepackage[dvipsnames,table,xcdraw]{xcolor}
\usepackage{fullpage}
\usepackage{stmaryrd}
\usepackage{enumitem}
\usepackage{mathtools}
\usepackage{microtype}
\usepackage{todonotes,booktabs}
\usepackage{float}

\usepackage{diagbox}

\newcommand{\vsim}{
\rotatebox[origin=c]{270}{$\sim$}
}

 \newcommand{\pullbackcorner}[1][ul]{\save*!/#1-1.25pc/#1:(-1,1)@^{|-}\restore}
  \newcommand{\C}{\mathcal{C}} 
  \newcommand{\Map}{\mathrm{Map}}

\newcommand{\pp}{\mathfrak{p}}

\newcommand{\Cat}{\mathsf{Cat}}
\newcommand{\End}{\mathrm{End}}
  
\usepackage{tikz}
\usetikzlibrary{positioning,arrows}
\usetikzlibrary{shapes.geometric, calc}
\usetikzlibrary{decorations.pathreplacing}

\newtheorem{theorem}{Theorem}[section]
\newtheorem{corollary}[theorem]{Corollary}
\newtheorem{lemma}[theorem]{Lemma}
\newtheorem{proposition}[theorem]{Proposition}
\newtheorem*{theorem*}{Theorem}

\theoremstyle{definition}
\newtheorem{definition}[theorem]{Definition}
\newtheorem{example}[theorem]{Example}
\newtheorem{remark}[theorem]{Remark}
\newtheorem{construction}[theorem]{Construction}
\usepackage{hyperref}

\definecolor{dark-red}{rgb}{0.5,0.15,0.15}
\definecolor{dark-blue}{rgb}{0.15,0.15,0.6}
\definecolor{dark-green}{rgb}{0.15,0.6,0.15}
\hypersetup{
    colorlinks, linkcolor=dark-red,
    citecolor=dark-blue, urlcolor=dark-green
}

\definecolor{gRed}{HTML}{ff5100}
\definecolor{gGreen}{HTML}{2b83ba}

\newcommand*\circled[1]{\tikz[baseline=(char.base)]{
            \node[shape=circle,draw,inner sep=2pt] (char) {#1};}}

\title{Model structures on finite total orders}
\date{\today}

\author{Scott Balchin}
\address[Balchin]{Max Planck Institute For Mathematics, Vivatsgasse 7, 53111 Bonn, Germany}
\email{balchin@mpim-bonn.mpg.de}

\author{Kyle Ormsby}
\address[Ormsby]{Department of Mathematics, Reed College, Portland, OR 97202, USA}
\email{ormsbyk@reed.edu}

\author{Ang\'{e}lica M.~Osorno}
\address[Osorno]{Department of Mathematics, Reed College, Portland, OR 97202, USA}
\email{aosorno@reed.edu}

\author{Constanze Roitzheim}
\address[Roitzheim]{School of Mathematics, Statistics and Actuarial Science, University of Kent, Kent CT2 7FS, UK}
\email{c.roitzheim@kent.ac.uk}

\begin{document}

\begin{abstract}
We initiate the study of model structures on (categories induced by) lattice posets, a subject we dub \emph{homotopical combinatorics}. In the case of a finite total order $[n]$, we enumerate all model structures, exhibiting a rich combinatorial structure encoded by Shapiro's Catalan triangle. This is an application of previous work of the authors on the theory of $N_\infty$-operads for cyclic groups of prime power order, along with new structural insights concerning extending  choices of certain model structures on subcategories of $[n]$.
\end{abstract}

\maketitle

\tableofcontents
  
\section{Introduction}

A Quillen model structure is a framework in which one can perform abstract homotopy theory within a category of interest. At its core, the theory of model structures is designed to solve the problem of formally inverting a chosen class of morphisms $\mathsf{W}$ called \emph{weak equivalences} in a category $\mathcal{C}$. It achieves this through the use of extra structure coming from a choice of \emph{fibrations} and \emph{cofibrations} which satisfy various topologically--inspired compatibility axioms. Indeed, potentially the most familiar example of a model structure is the one on the  category $\mathbf{Top}$ of topological spaces and continuous maps where the weak equivalences are the weak homotopy equivalences, the fibrations are the Serre fibrations, and the cofibrations are those maps which are retracts of relative cell complexes.

Fix a category $\mathcal{C}$. It is natural to ask if we can enumerate or determine structural properties of the collection $Q(\mathcal{C})$ of all Quillen model structures on $\mathcal{C}$. It turns out that the behaviour of $Q(\mathcal{C})$ can be wildly different based on the choice of $\mathcal{C}$. Even when fixing the weak equivalences, we can exhibit a range of behaviours.

To see this, consider the category $\textbf{sSet}$ of simplicial sets and natural transformations between them. The weak equivalences in the standard Kan--Quillen model structure on $\textbf{sSet}$ are those morphisms $f \colon X \to Y$ such that the geometric realization $|f| \colon |X| \to |Y|$ is a weak homotopy equivalence of spaces. The fibrations are the \emph{Kan fibrations} and the cofibrations are the monomorphisms~\cite{quillen}. One may wonder if there is another model structure on $\textbf{sSet}$ with these weak equivalences but with different fibrations and cofibrations. Beke answers this question in the positive by proving the existence of an infinite collection of such model structures~\cite{beke}.

On the other end of the spectrum, consider the category $\textbf{Cat}$ of small categories and functors between them. The natural choice of weak equivalences being equivalences of categories leads to a model structure on $\textbf{Cat}$ where the fibrations are the isofibrations and the cofibrations are the functors which are injective on objects~\cite{rekz}. Amazingly, this is the only model structure on $\textbf{Cat}$ with these weak equivalences, in stark contrast to the case of simplicial sets~\cite{canonical}.

Therefore, we have seen that even when fixing the weak equivalences, there could be a whole range of behaviours for choices of fibrations and cofibrations. The only classification result for $Q(\mathcal{C})$ that we are aware of is the full description of model structures on the category $\mathbf{Set}$ of sets and functions between them. One can explicitly prove that there are exactly nine model structures on $\mathbf{Set}$, that is, $\# Q (\mathbf{Set}) = 9$~\cite{tobyomar}.

In this paper, we initiate the systematic study of model structures on categories induced by lattices. Here a lattice $P$ is a partially ordered set admitting finite meets and joins, and the induced category $\mathcal P$ has objects $P$ and a unique morphism $p\to q$ if and only if $p\le q$ in $P$. Since model structures encode homotopy theories and lattices are of fundamental interest in combinatorics, one can think of our work as an initial investigation into \emph{homotopical combinatorics}.\footnote{At the same time, \emph{combinatorial homotopy theory} might be an appropriate appellation since we will expose the rich combinatorial structure undergirding model structures on lattices. Beware, though, that \emph{combinatorial model structures} already exist in the literature and have a rather different flavor.}

Our main result is a full classification of model structures on finite total orders. That is, for a fixed $n \in \mathbb{N}$ we consider the poset $[n]=\{0<1 < \dots <n\}$ as a category, and classify the structure of all model structures on $[n]$. As $[n]$ has finitely many objects and finitely many morphisms, it follows that $Q([n])$ is a finite set and as such can be enumerated. We prove the following.

\begin{theorem*}[Theorems ~\ref{thm:Qn} and ~\ref{thm:refined}]
Let $Q([n])$ be the collection of Quillen model structures on $[n]$. Then
\[
\# Q([n]) = \binom{2n+1}{n}
\]
and, for $0\le k\le n$, exactly
\[
\frac{2(k+1)}{n+k+2}\binom{2n+1}{n-k}
\]
of these model structures have homotopy category isomorphic to $[k]$.
\end{theorem*}

This enumeration result follows from a surprising link to the theory of $N_\infty$-operads, which are a tool from equivariant homotopy theory concerning the classification of different types of homotopy commutativity. The first and fourth authors, along with D. Barnes, proved an enumeration result for $N_\infty$-operads up to weak equivalence for the cyclic groups $C_{p^n}$. In particular it was shown that
\[
\# N_{\infty}(C_{p^n}) = \mathsf{Cat}(n+1) = \frac{1}{n+2} \binom{2n+2}{n+1}
\]
where $\mathsf{Cat}(n+1)$ is the $(n+1)$-th Catalan number. It was then observed by the second and third author and collaborators in~\cite{fooqw} that there is a bijection between the set of weak equivalence classes of $N_\infty$-operads on $C_{p^n}$ and the set of model structures on $[n] = \operatorname{Sub}(C_{p^n})$ where all morphisms are weak equivalences. This is the key insight that led to the development of this paper.

The second main result appearing in this paper is a full description of the Bousfield lattice of $[n]$. Starting from any model structure, one can consider adding more weak equivalences. If this is done in a way that preserves the cofibrations (resp., fibrations) then the process is called left (resp., right) Bousfield localization. 

In general it is a very hard problem to decide which model structures can be obtained from other others through a process of left and right localizations. In the case of $[n]$, we obtain a full classification. 

\begin{theorem*}[Theorem~\ref{thm:bousfieldlocn}]
Every model structure on $[n]$ can be obtained via a sequence of left and right Bousfield localizations starting at the trivial model structure, where only the isomorphisms are weak equivalences and all maps are fibrations and cofibrations.
\end{theorem*}

Alongside these theorems, we prove a plethora of results regarding model structures on general finite lattices and $[n]$ including the relation to premodel structures and the existence of an involution on the collection of model structures.

\subsection*{Relation to other work}
Model structures on posets have been studied in a non-enumerative fashion by Droz and Zakharevich \cite{dz21}. While we make use of several of their structural results in Section~\ref{sec:prelims}, their emphasis (proving that extending to a model structure is not a first-order property) is quite different from ours.

We also note that the homotopy theory of the \emph{category of posets} has been studied in \cite{raptis}. Our focus on the homotopy theory of \emph{poset categories} is only similar linguistically.

\subsection*{Outline}

In Section~\ref{sec:prelims} we recall the relevant information regarding model categories, with a particular focus on reductions that can be done when the underlying category is a finite lattice.

The main structural result appears in Section~\ref{sec:extending}, which proves that to uniquely define a model structure, it is enough to define a collection of \emph{contractible submodels}. This immediately leads to the aforementioned enumeration result which is the subject of Section~\ref{sec:enum}.

In Section~\ref{sec:bousfield}, we prove that every model structure on $[n]$ can be obtained through a zig-zag of left and right Bousfield localizations starting from the trivial model structure. Finally, in Section~\ref{sec:further}, we provide a list of potential directions for further exploration.

\subsection*{Acknowledgements}

The first author would like to thank the Max Planck Institute for Mathematics for its
hospitality. The second and third authors were supported by NSF grant DMS--1709302. The fourth author would like to thank the London Mathematical Society for an Emmy Noether Fellowship. The authors thank Michael Kinyon and Federico Ardila for a helpful discussion regarding the history of Lemma \ref{lemma:CpId}, and the anonymous referee for helpful input.

\section{Recollections on model categories}\label{sec:prelims}

\subsection{Model categories}\label{subsec:modelcats}

We begin by  recalling some basic definitions and conventions about model categories. Many readers may be familiar with these, but it will be beneficial to collect the necessary tools for our results in one place.  We refer the reader to~\cite{hovey,dwyerspa} or~\cite{handbook} for further details, as well as the original reference of~\cite{quillen}.

\begin{definition}
For any two morphisms $i \colon A \to B$ and $p \colon X \to Y$ in a category $\mathcal{C}$, we say that $i$ \emph{has the left lifting property (LLP) with respect to $p$}, or $p$ \emph{has the right lifting (RLP) property with respect to $i$},
if for all commutative squares of the form
\begin{equation}\label{liftdiagram}
\begin{gathered}\xymatrix{
A \ar[r] \ar[d]_{i} & X \ar[d]^{p} \\ B \ar[r] & Y
}\end{gathered}    
\end{equation}
there exists a lift $h \colon B \to X$ which makes the resulting diagram commute. If $i$ lifts on the left of $p$ we write $i \boxslash p$. For any class $\mathcal{S}$ of morphisms in $\mathcal{C}$ we write
\begin{align*}
\mathcal{S}^{\boxslash}& = \{g \in \operatorname{Mor}(\mathcal{C}) \mid f \boxslash g \text{ for all } f \in \mathcal{S} \},\\
{}^{\boxslash} \mathcal{S} &= \{f \in \operatorname{Mor}(\mathcal{C}) \mid f \boxslash g \text{ for all } g \in \mathcal{S} \}.
\end{align*}
Note that $\mathcal{S} \subseteq {}^\boxslash \mathcal{T}$ if and only if $\mathcal{T} \subseteq \mathcal{S}^\boxslash$. We write $\mathcal{S}\boxslash \mathcal{T}$ when this holds.
\end{definition}

\begin{definition}\label{defn:modelcat}
A \emph{model category} is a category $\mathcal{C}$ equipped with three distinguished classes of morphisms, namely
\begin{itemize}
    \item weak equivalences -- $\mathsf{W}$ -- whose elements we will represent as $\xymatrix{X \ar[r]^{\sim} & Y}$,
    \item fibrations -- $\mathsf{F}$ -- whose elements we will represent as $\xymatrix{X \ar@{->>}[r] & Y}$,
    \item cofibrations -- $\mathsf{C}$ -- whose elements we will represent as $\xymatrix{X \ar@{^(->}[r] & Y}$,
\end{itemize}
each of which is closed under composition. A morphism in $\mathsf{AF}:=\mathsf{W} \cap \mathsf{F}$ (resp., $\mathsf{AC}:=\mathsf{W} \cap \mathsf{C}$) is said to be an \emph{acyclic fibration} (resp., \emph{acyclic cofibration}). These distinguished classes of morphisms and the category $\mathcal{C}$ are required to satisfy the following axioms.
\begin{enumerate}[align=left]
    \item[MC1)] The category $\mathcal{C}$ has all finite limits and colimits. In particular, there is an initial object $\varnothing$ and terminal object $\ast$.
    \item[MC2)] The class $\mathsf{W}$ satisfies the two-out-of-three property.
    \item[MC3)] The three distinguished classes of morphisms are closed under retracts in the arrow category.
    \item[MC4)] Given a commutative diagram of the form (\ref{liftdiagram}) above, a lift exists when either $i$ is a cofibration and $p$ is an acyclic fibration, or when $i$ is an acyclic cofibration and $p$ is a fibration.
    \item[MC5)] Each morphism $f$ in $\mathcal{C}$ an be factored in two ways:
    \begin{enumerate}
        \item[1.] $f=pi$, where $i$ is a cofibration and $p$ is an acyclic fibration.
        \item[2.] $f=pi$, where $p$ is a fibration and $i$ is an acyclic cofibration.
    \end{enumerate}
\end{enumerate}

If objects $X$ and $Y$ lie in the same weak equivalence class then we shall write $X \simeq Y$.
\end{definition}

\begin{lemma}\label{lem:wafac}
Let $\mathcal{C}$ be a model category. Then $\mathsf{W} = \mathsf{AF} \circ \mathsf{AC}$ (i.e., those maps that can be written as the composition of an acyclic cofibration and an acyclic fibration).
\end{lemma}

\begin{proof}
As the weak equivalences are closed under composition, we have that any element of $\mathsf{AF} \circ \mathsf{AC}$ is in $\mathsf{W}$. Conversely, let $f \in \mathsf{W}$. Then by the factorization axiom we can write $f = pi$ where $p$ is an acyclic fibration and $i$ is a cofibration. By the 2-out-of-3 property of weak equivalences it follows that $i$ is is a weak equivalence, and hence an acyclic cofibration as required. 
\end{proof}

\begin{remark}\label{rem:overdetermined}
It follows from the axioms of a model category along with the above lemma that the data of a model structure is over-determined, as, for example, the cofibrations can be recovered using only the weak equivalences and fibrations via lifting properties. Explicitly, we have ${}^\boxslash \mathsf{AF} = \mathsf{C}$, \emph{i.e.}, the cofibrations are exactly those morphisms with the LLP with respect to all acyclic fibrations (and the dual statement for fibrations).

Following similar reasoning, it is sufficient to provide just the data of the cofibrations and weak equivalences, which then uniquely determines the fibrations. We will repeatedly use this fact as we will naturally define our model structures via their weak equivalences and acyclic fibrations.
\end{remark}

\begin{definition}
Let $\mathcal{C}$ be a model category. An object $X$ in $\mathcal{C}$ is said to be
\begin{itemize}
    \item \emph{fibrant} if the unique morphism $X \to \ast$ is a fibration,
    \item \emph{cofibrant} if the unique morphism $\varnothing \to X$ is a cofibration,
    \item \emph{bifibrant} if it is both fibrant and cofibrant.
\end{itemize}
\end{definition}

For any object $X$ in a model category $\mathcal{C}$ it is possible to approximate $X$ --- up to weak equivalence --- by a (co)fibrant object using the factorization axiom. In particular, for every $X$ there is a choice of
\begin{itemize}
    \item An object $X^{\mathsf{f}}$ and an acyclic cofibration $\xymatrix{X \ar@{^(->}[r]^{\sim}& X^{\mathsf{f}}}$. We say $X^{\mathsf{f}}$ is a \emph{fibrant replacement} of $X$.
    \item An object $X^{\mathsf{c}}$ and an acyclic fibration $\xymatrix{X^{\mathsf{c}} \ar@{->>}[r]^{\sim}& X}$. We say $X^{\mathsf{c}}$ is a \emph{cofibrant replacement} of $X$.
\end{itemize}

The selling point for model categories is that they provide a framework in which to perform localizations of categories with respect to a class of weak equivalences, and the bifibrant objects play a pivotal role in this. In particular, given a model category $\mathcal{C}$, one wishes to form the universal category $\mathcal{C}[\mathsf{W}^{-1}]$ in which the weak equivalences have been formally inverted. This is a model for the the \emph{homotopy category} as we now define.

%The \emph{homotopy category}, which we presently introduce, plays this role.

\begin{definition}\label{def:homotopycat}
Let $\mathcal{C}$ be a model category. Then its \emph{homotopy category} $\operatorname{Ho}(\mathcal{C})$ is the category $\mathcal{C}[\mathsf{W}^{-1}]$ obtained from $\mathcal{C}$ by formally inverting the weak equivalences.
%whose objects are the bifibrant objects of $\mathcal{C}$, and whose morphisms are equivalence classes of maps up to homotopy (see the subsequent paragraph).
\end{definition}

\begin{remark}
    In general, given a category $\mathcal{C}$, and a class of morphisms $\mathsf{W}$ in $\mathcal{C}$, formally inverting $\mathsf{W}$ leads to requiring zig-zags of morphisms of potentially arbitrary length, so one does not obtain a category with small hom-sets. A key feature of the theory of model categories is that the homotopy category can be constructed using zig-zags of length most two, so Definition~\ref{def:homotopycat} is well defined.
\end{remark}

%We are aware that we have not actually introduced the notion of homotopy between morphisms in a model category. As this is rather lengthy, let us just say the following: Between bifibrant objects a map has a homotopy inverse if and only if it is a weak equivalence. Therefore, weak equivalences become isomorphisms in the homotopy category. This means that $\Ho(\C)$ is equivalent to $\C[W^{-1}]$, but not isomorphic. 

We have, on occasion, reason to want to compare model structures. That is, we would like a notion of functors $f \colon \mathcal{C} \to \mathcal{D}$ between model categories which descend to a functor between the associated homotopy categories. These are exactly the \emph{Quillen functors} as we now define. The fact that they have the desired properties can be found in, for example,~\cite[Section 1.3]{hovey}.

\begin{definition}
Let $\mathcal{C}$ and $\mathcal{D}$ be model categories. An adjoint pair of functors
\[
F :\mathcal{C} \rightleftarrows \mathcal{D} : U
\]
is a \emph{Quillen pair} if the left adjoint $F$ preserves cofibrations and the right adjoint $U$ preserves fibrations. We say that $F$ is a \emph{left Quillen functor} and $U$ is a \emph{right Quillen functor}.
\end{definition}

We finish this section by highlighting a helpful fact regarding the behaviour of isomorphisms in a model category. This will be of use when discussing model structures on lattices where the only isomorphisms are the identity morphisms.

\begin{lemma}\label{lem:isomorphism}
Let $\mathcal{C}$ be a model category. Then a morphism $f \colon X \to Y$ is an isomorphism if and only if it is in all three classes of maps $\mathsf{W}$, $\mathsf{F}$ and $\mathsf{C}$.
\end{lemma}

\begin{proof}
Suppose $f \colon X \to Y$ is in all three classes of maps. Then we can form a commutative diagram
\[\xymatrix{
    X \ar[r]^{\mathrm{id}_X} \ar[d]_f & X \ar[d]^f \\
    Y \ar[r]_{\mathrm{id}_Y} & Y.
}\]
\emph{A priori} the left vertical map is an acyclic cofibration, and the right vertical map is a fibration. Thus there is a lift $h \colon Y \to X$ rendering the two triangles commutative, that is, $f \circ h = \mathrm{id}_Y$ and $h \circ f = \mathrm{id}_X$, showing that $f$ is an isomorphism as required.

For the converse statement, assume that $f \colon X \to Y$ is an isomorphism with inverse $f^{-1}$. Then we note we can solve any lifting problems of the form
\[\xymatrix{
X \ar[d]_f \ar[r] & A \ar[d] & & A \ar[d] \ar[r] & X \ar[d]^f \\
Y \ar[r] & B & & B \ar[r] & Y
}\]
simply by making use of the inverse $f^{-1}$. In particular, $f \colon X \to Y$ lifts on the left with respect to any fibration (resp., lifts on the right with respect to any cofibration) and is therefore an acyclic cofibration (resp., acyclic fibration). As such, $f$ is in all three classes of maps.
\end{proof}

\subsection{Weak factorization systems}

Although Definition~\ref{defn:modelcat} is the definition of a model category that one will usually see in the literature, for the purposes of this article it will be more convenient to instead consider model categories through the lens of \emph{weak factorization systems} and \emph{premodel structures}.

\begin{definition}
Let $\mathcal{L}$ and $\mathcal{R}$ be classes of morphisms in a category $\mathcal{C}$. Then the pair $(\mathcal{L},\mathcal{R})$ is a \emph{weak factorization system} if
\begin{enumerate}
\item every morphism $f$ in $\mathcal{C}$ can be factored as $f=pi$ with $i \in \mathcal{L}$ and $p \in \mathcal{R}$,
\item $\mathcal{L} \boxslash \mathcal{R}$,
\item $\mathcal{L}$ and $\mathcal{R}$ are closed under retracts.
\end{enumerate}
\end{definition}

It is immediate from the definitions that every model structure on $\C$ gives rise to two weak factorization systems, $(\mathsf{C},\mathsf{AF})$ and $(\mathsf{AC},\mathsf{F})$.  Conversely, one sees that model categories can be completely characterized in terms of weak factorization systems, as was noted in \cite{JT}.
\begin{theorem}\label{thm:equivdef}
Let $\mathcal{C}$ be a category with all finite limits and colimits, and let $\mathsf{W}$, $\mathsf{C}$, and $\mathsf{F}$ be three classes of morphisms. Then $\mathsf{W}$, $\mathsf{C}$, and $\mathsf{F}$ determine a model category structure on $\mathcal{C}$ if and only if
\begin{enumerate}
    \item $\mathsf{W}$ satisfies the 2-out-of-3 property, and
    \item $(\mathsf{C},\mathsf{AF})$ and $(\mathsf{AC},\mathsf{F})$ are weak factorization systems on $\mathcal{C}$.
\end{enumerate}
\end{theorem}

The idea of \emph{premodel structures} as introduced in~\cite{barton} (which are also called \emph{Quillen structures} in~\cite{quillenstruct}) is that asking for the 2-out-of-3 property for the weak equivalences is counter intuitive in certain situations such as when taking (co)limits of model categories. Instead, one may just ask for a compatible pair of weak factorization systems. 

\begin{definition}\label{def:premodel}
A \emph{premodel category} is a category $\mathcal{C}$ with all finite limits and colimits equipped with four classes of maps $\mathsf{C}$, $\mathsf{AC}$, $\mathsf{F}$ and $\mathsf{AF}$ such that the pairs $(\mathsf{C},\mathsf{AF})$ and $(\mathsf{AC},\mathsf{F})$ are weak factorization systems on $\mathcal{C}$, and $\mathsf{AC} \subseteq \mathsf{C}$ (equivalently, $\mathsf{AF} \subseteq \mathsf{F}$).
\end{definition}

Note that any model structure is a premodel structure, as follows directly from Lemma~\ref{lem:wafac} and Theorem~\ref{thm:equivdef}, but the converse does not hold \emph{ex consilio}.

\subsection{Model structures on lattices}

In \S\ref{subsec:modelcats} we have introduced model categories for general categories $\mathcal{C}$. However, our focus in this article will be the category arising from the poset $[n]$, \emph{i.e.}, the poset whose objects are $\{0,1,\dots, n\}$ and with order relation $i \leq j$ inherited from the ordering of $\mathbb{N}$. 

As we require our categories to have all small limits and colimits, we are therefore only interested in those posets with all limits and colimits. These posets are complete lattices, and any finite lattice is complete. As such we will simply call a finite complete and cocomplete poset a lattice. We recall that the limits are computed via meets $\wedge$ (\emph{i.e.}, the greatest lower bound) and colimits are computed via joins $\vee$ (\emph{i.e.}, the least upper bound). In the case of $[n]$ the meet operation corresponds to $\min$ while the join operation corresponds to $\max$.

We will now collect several relevant reductions to the basic properties of model structures that can be made in the case that $\mathcal{C} = \mathcal{P}$ is a lattice from~\cite{dz21}. 

Explicitly, by a \emph{lattice} we mean a skeletal category $\mathcal{P}$ such that for all objects $X$ and $Y$, $\#\operatorname{Hom}_\mathcal{P}(X,Y) \leq 1$ which admits all finite limits and colimits. Given a classical lattice $P$, we define a category $\mathcal{P}$ with $\operatorname{ob}(\mathcal{P}) = P$ and
\[
\operatorname{Hom}_\mathcal{P}(X,Y)=
\begin{cases}
\{\ast\}& \text{if } X \leq Y,\\
\, \, \, \varnothing & \text{else}.
\end{cases}
\]
Note that in a lattice there are no non-trivial retracts so axioms regarding retracts are vacuous in this setting. 

The first result regarding model structures on lattices is that the weak equivalences necessarily satisfy a stronger property than 2-out-of-3. We begin with a definition.

\begin{definition}
A class $\mathcal{E}$ of morphisms in a category $\mathcal{C}$ is \emph{decomposable} if $f \in \mathcal{E}$ with $f=gh$ for some $g,h\in \mathrm{Mor}(\mathcal C)$ implies that both $g$ and $h$ are in $\mathcal{E}$.
\end{definition}

\begin{lemma}[{\cite[Proposition 1.8]{dz21}}]
\label{lem:decomp}
Let $\mathcal{P}$ be a lattice. Then the weak equivalences of any model structure on $\mathcal{P}$ are decomposable.
\end{lemma}

The second result that we will make use of is a restriction on the number of bifibrant objects in any given weak equivalence class.

\begin{lemma}[{\cite[Lemma 1.6]{dz21}}]
\label{lem:uniquebif}
Let $\mathcal{P}$ be a lattice equipped with a model structure. Then each weak equivalence class has a unique bifibrant object.
\end{lemma}

In the case of a countable lattice, the homotopy category as defined in Definition~\ref{def:homotopycat} admits a much simpler description.

\begin{lemma}[{\cite[Theorem 5.4]{dz21}}]\label{lem:homotopycategoryposet}
Let $\mathcal{P}$ be a countable lattice equipped with a model structure. Then $\operatorname{Ho}(\mathcal{P})$ is equivalent to the full subcategory on the bifibrant objects.
\end{lemma}

\begin{remark}
We highlight that Lemmas~\ref{lem:decomp}, \ref{lem:uniquebif} and \ref{lem:homotopycategoryposet} do not hold for an arbitrary model category, and the proofs critically use the fact that we are working in a lattice. For Lemma~\ref{lem:decomp} one can cook up a simple counter example in a category using retracts. For Lemmas~\ref{lem:uniquebif} and \ref{lem:homotopycategoryposet}
 it is enough to observe that there are model structures in which all objects are bifibrant, and such that the weak equivalences are more than just the isomorphisms. A classical example of such a model structure is the Str{\o}m model structure on $\mathbf{Top}$~\cite{strom}. 
 \end{remark}

We now introduce terminology for two extremes of model structures that we will encounter in subsequent sections.

\begin{definition}\label{def:contractible}
Let $\mathcal{C}$ be a complete and cocomplete category equipped with a model structure.
\begin{itemize}
    \item If all morphisms in $\mathcal{C}$ are weak equivalences, then we say that the model structure is \emph{contractible}.
    \item If only the isomorphisms in $\mathcal{C}$ are weak equivalences, then we say that the model structure is \emph{trivial}.
\end{itemize}
\end{definition}

\begin{remark}
We shall see that for an arbitrary lattice that there are many contractible model structures, however, there is only one trivial model structure on any given category. Indeed, if only the isomorphisms are weak equivalences, then it is an instructive exercise to check that all maps must both be fibrations and cofibrations so that the required factorizations exist.
\end{remark}

The terminology of Definition~\ref{def:contractible} is justified by the following consequence of Lemma \ref{lem:homotopycategoryposet}.

\begin{corollary}
Let $\mathcal{P}$ be a lattice equipped with a model structure.
\begin{itemize}
    \item If $\mathcal{P}$ is a contractible model structure then $\operatorname{Ho}(\mathcal{P})$ is equivalent to the category with a single object and only the identity morphism.
    \item If $\mathcal{P}$ is a trivial model structure then $\operatorname{Ho}(\mathcal{P})$ is equivalent to $\mathcal{P}$.
\end{itemize}
\end{corollary}

\begin{remark}\label{rem:contractibleiswfs}
Every weak factorization system $(\mathcal{L}, \mathcal{R})$ on an arbitrary category $\mathcal{C}$ determines a contractible model structure by letting
\begin{itemize}
    \item $\mathsf{W} = \mathrm{all}$,
    \item $\mathsf{F} = \mathcal{R}$,
    \item $\mathsf{C} = \mathcal{L}$.
\end{itemize}
In fact, every contractible model structure arises as such. In other words, there is a bijection of sets

\[
\xymatrix{
\{\text{weak factorization systems on }\mathcal{C}\} \ar@{<=>}[rr]& &
\{\text{contractible model structures on }\mathcal{C}\}}.
\]
\end{remark}

There is yet another way to classify the contractible model structures on $\mathcal{C}$ in the case that $\mathcal{C}$ is a finite lattice. We begin with a definition.

\begin{definition}\label{def:transfersystem}
%Let $\mathcal{P} = (\mathcal{P},\leq)$ be a lattice. 
A \emph{transfer system} in a category $\mathcal{C}$ is a wide subcategory of $\mathcal{C}$ which is closed under
pullbacks by arbitrary morphisms in $\mathcal{C}$.
%A \emph{transfer system} on $\mathcal{P}$ consists of a partial order $\mathcal{R}$ that refines $\leq$ and such that for all $z,y,z \in \mathcal{P}$, if $x \,\mathcal{R}\, y$, $z \leq y$, then $(x \wedge z) \, \mathcal{R} \, z$.
\end{definition}
\noindent (Here a subcategory of $\mathcal{C}$ is \emph{wide} when it contains all the objects of $\mathcal{C}$.)

The next result tells us that the data of a transfer system is exactly the data of a weak factorization system when $\mathcal{P}$ is a finite lattice. In particular, one needs to only provide the right (or left) set. In the following result, we consider poset structures on transfer systems and weak factorization systems. The order on transfer systems is by inclusion, and the order on weak factorization systems is given by inclusion of the right class.  %That is, we say $\mathcal{R}_1 \leq \mathcal{R}_2$ if and only if  for all $x,y \in \mathcal{P}$, if $x \,\mathcal{R}_1\, y$ then $x \,\mathcal{R}_2\, y$. The order on weak factorization systems is given by inclusion of the right class. %\marginpar{Angelica: Since we are defining transfer systems as subcategories, we might as well define the relation as inclusion, and forgo talking about refinement here.}

\begin{proposition}[{\cite[Proposition 4.11, Theorem 4.13]{fooqw}}]\label{fooqwresult}
Let $\mathcal{P}$ be a finite lattice and $\mathcal{R}$ a transfer system on $\mathcal{P}$. Then there is a unique weak factorization system $(\mathcal{L},\mathcal{R})$ on $\mathcal{P}$. In particular the assignment
\[
\mathcal{R} \longleftrightarrow ({}^\boxslash \mathcal{R}, \mathcal{R})
\]
is an isomorphism between the poset of transfer systems on $\mathcal{P}$ and the poset of weak factorization systems on $\mathcal{P}$. 
\end{proposition}

This result is powerful as not only does it give a simple characterization of the right class, the left class admits an explicit combinatorial description as we now recall.

\begin{lemma}[{\cite[Proposition 4.15]{fooqw}}]
For a transfer system $\mathcal{R}$ on a lattice $\mathcal{P}$, the set $\mathcal{L}$ is given by $\mathcal{E}({R})^c$. Here, $\mathcal{E}({R})$ is the downward extension of $\mathcal{R}$, defined as
\[
\mathcal{E}({R})=\{ z \rightarrow y \,\,|\,\,\mbox{there exists $x \in \mathcal{P}$ such that $z \leq x < y$ and $x \rightarrow y \in \mathcal{R}$ } \}
\]
and $\mathcal{E}(R)^c$ is the complement of $\mathcal{E}(R)$ in $\mathcal{P}$.
\end{lemma}

To conclude, in the situation of a finite lattice $\mathcal{P}$ we have the following triple of bijections

\[
\xymatrix@C=-2em{
& \{\text{transfer systems on }\mathcal{P}\ar@{<=>}[dr] \ar@{<=>}[dl]\} \\
\{\text{weak factorization systems on }\mathcal{P}\} \ar@{<=>}[rr]& &
\{\text{contractible model structures on }\mathcal{P}\}.}
\]

\begin{example}
Consider the lattice $[1] := \xymatrix@C=4em{\bullet \ar[r] & \bullet}$. To determine a model structure on this lattice, we need only determine the status of the single non-identity morphism, that is, decide if it is in $\mathsf{W}$, $\mathsf{C}$ or $\mathsf{F}$. By Lemma~\ref{lem:isomorphism} we know that this map can only be in two out of the three distinguished classes of maps (as the map is not an isomorphism). 

The non-trivial morphism has to be either a fibration or cofibration (or both), as otherwise (MC5) would not be satisfied. If it is not a weak equivalence, we are in the trivial model structure, so then it has to be both a fibration and a cofibration at the same time.

By \cite[Example 1.1.5]{hovey} for any category $\mathcal{C}$ admitting all limits and colimits we always have the existence of the following model structures in addition to the trivial one.
\begin{itemize}
    \item $\mathsf{F} = \mathrm{iso}$, $\mathsf{W} = \mathsf{C} = \mathrm{all}$, 
    \item $\mathsf{C} = \mathrm{iso}$, $\mathsf{F} = \mathsf{W} = \mathrm{all}$.
\end{itemize}
And indeed, we have seen that these are all the cases taken care of, so the total number of possible model structures on $[1]$ is 3, where one is trivial and the other two are contractible. We can represent these model structures as follows.
\[
\xymatrix@C=4em@R=0em{
\bullet \ar@{^(->}[r]^{\sim} & \bullet & & \bullet \ar@{^(->>}[r] & \bullet & & \bullet \ar@{->>}[r]^{\sim}& \bullet \\
\ar@{}[r]|{\mathrm{contractible}} & & & \ar@{}[r]|{\mathrm{trivial}}  & & & \ar@{}[r]|{\mathrm{contractible}}  & 
}
\]

\end{example}

\section{Extending contractible submodels}\label{sec:extending}

In this section we will prove a structural result regarding model structures on the lattice $[n]$ which will lead to the desired enumeration result in the next section.  The main insight is that it is enough to define certain contractible submodel structures which then uniquely determine a model structure on $[n]$. We will repeatedly use the fact that in the lattice $[n]$ the pushout of the span
\[
\xymatrix{X &\ar[l] Z \ar[r]& Y}
\]
is computed as $\max(X,Y)$. Dually the pullback of the cospan
\[
\xymatrix{X &\ar@{<-}[l] Z \ar@{<-}[r]& Y}
\]
is computed as $\min(X,Y)$.

\subsection{Contractible submodels}

We begin by determining the possible structure of weak equivalences in a model structure on $[n]$. By the fact that the weak equivalences are decomposable, if the morphism $X \to Y$ is a weak equivalence, then so are the morphisms $X \to Z$ and $Z \to Y$ for all $X \leq Z \leq Y$. As such, the weak equivalence classes determine and are determined by an interval partition of $[n]$ as we will now define.

\begin{definition}
An \emph{interval partition} of $[n]$ is a partition $\mathfrak{p}$ of $[n]$ into intervals of form
\[
[0,a_1] \amalg [a_1+1,a_2] \amalg \cdots \amalg [a_k+1,n]
\]
where we allow $a_i+1 = a_{i+1}$.
\end{definition}

\begin{remark}\label{rmk:partition}
It should be noted that the interval partitions of $[n]$ are naturally in bijection with ordered partitions (or \emph{compositions}) of the integer $n+1$. As such, there are exactly $2^n$ such interval partitions. We have chosen to work with interval partitions as it makes the link to the lattice $[n]$ more apparent.
\end{remark}

The following result is immediate from the decomposability and 2-out-of-3 property for the weak equivalences.

\begin{lemma}
There is a bijection between weak equivalence structures on $[n]$ and interval partitions on $[n]$.
\end{lemma}

We are now in a position to define what we mean by a selection of \emph{contractible submodels}.

\begin{definition}
Let $[n]$ be equipped with an interval partition $\mathfrak{p} = [0,a_1] \amalg [a_1+1,a_2] \amalg \cdots \amalg [a_k+1,n]$. Then a choice of \emph{contractible submodels} on $[n]$ (relative to $\pp$) is a choice of a contractible model structure on each interval $[a_i+1,a_{i+1}] \cong [r_i]$ where $r_i=a_{i+1}-a_i-1$ and $i=1,\dots,k+1$.
\end{definition}

A choice of contractible submodels on $[n]$ therefore picks the classes $\mathsf{W}$, $\mathsf{AF}$ and $\mathsf{AC}$ for a unique (potential) model structure on $[n]$ (\emph{cf.}, Remark~\ref{rem:overdetermined}). In the next section we will show that this indeed always defines a model structure. For now we observe that any model structure on $[n]$ gives rise to a choice of contractible submodels by restriction.

\begin{proposition}\label{prop:modelgivessub}
Any model structure on $[n]$ induces a model structure on each weak equivalence class.
%Let $[n]$ be equipped with a model structure. Then there is a unique choice of contractible submodels on $[n]$ compatible with weak equivalences in the model structure.
\end{proposition}

\begin{proof}
Let $\mathfrak{p} = [0,a_1] \amalg [a_1+1,a_2] \amalg \cdots \amalg [a_k+1,n]$ be the interval poset for the weak equivalence structure associated to the model structure on $[n]$. We need to show that the restriction of the model structure to each interval $[a_i+1,a_{i+1}]$ provides a contractible submodel. By construction all maps in this restricted model structure are weak equivalences, as such it suffices to show that the restriction to the interval gives a weak factorization system. The factorization of a morphism $f \colon X \to Y$  in $[a_i+1,a_{i+1}]$ follows from the factorizations in the model structure in $[n]$. Similarly, we note that for a commutative diagram of the form
\[
\begin{gathered}\xymatrix{
A \ar[r] \ar[d]_{i} & X \ar[d]^{p} \\ B \ar[r] & Y
}\end{gathered}    
\]
if both $i$ and $p$ are in the interval $[a_i+1,a_{i+1}]$ then so is the lift $B \to X$. As such, the liftings are also taken care of by the model structure on $[n]$ as required.
\end{proof}

\subsection{The existence result}

\begin{definition}
Let $[n]$ be equipped with a contractible model structure, and $X \in [n]$. We denote by $R^\mathrm{max}(X)$ the necessarily unique maximal object such that $X \to R^\mathrm{max}(X)$ is an acyclic cofibration. Similarly, denote by $Q^\mathrm{min}(X)$ the unique minimal object such that $Q^\mathrm{min}(X) \to X$ is an acyclic fibration.
\end{definition}

The following lemma highlights the choice of notation as being a maximal fibrant (resp., minimal cofibrant) replacement. In particular this gives a functorial choice of fibrant and cofibrant replacement.

\begin{lemma}
For every $X \in [n]$ as above, the object $R^\mathrm{max}(X)$ is fibrant. Dually, the object $Q^\mathrm{min}(X)$ is cofibrant.
\end{lemma}

\begin{proof}
We need to show that for any commutative diagram of the form
\[\xymatrix{
A \ar[r] \ar@{^(->}[d]_{\vsim} & R^{\mathrm{max}} (X) \ar[d] \\
B \ar[r] & \ast
}\]
there is a lift $B \to R^\mathrm{max}(X)$. Assume that there is not. Then there is a map $R^\mathrm{max}(X) \to B$, and we can construct a pushout
\[\xymatrix{
A \ar[r] \ar@{^(->}[d]_{\vsim}  & R^{\mathrm{max}} (X) \ar[d] \\
B \ar@{=}[r] & B \pullbackcorner
}\]
As pushouts of acyclic cofibrations are acyclic cofibrations, we have that $R^\mathrm{max}(X) \to B$ is an acyclic cofibration which contradicts the maximality of $R^\mathrm{max}(X)$.

A dual argument using that pullbacks of acyclic fibrations are acyclic fibrations gives a proof that $Q^\mathrm{min}(X)$ is cofibrant.
\end{proof}

\begin{corollary}\label{cor:cofibrep}
For every $X \in [n]$ the object $R^{\mathrm{max}}(X)$ is a functorial choice of fibrant replacement in a contractible model structure. Dually, the object $Q^{\mathrm{min}}(X)$ is a functorial choice of cofibrant replacement in a contractible model structure.\hfill\qedsymbol
\end{corollary}

\begin{remark}
It follows from the above corollary that the object $R^{\mathrm{max}}(Q^{\mathrm{min}}(X)) = Q^{\mathrm{min}}(R^{\mathrm{max}}(X))$ picks out the unique bifibrant object in a contractible model structure on $[n]$.
\end{remark}

We are now equipped to prove the main result of this section.

\begin{theorem}\label{thm:main}
Any selection of contractible submodels on $[n]$ can be uniquely extended to a model structure on $[n]$.
\end{theorem}

\begin{proof}
The selection of contractible submodels gives us a set $\mathsf{AF}$ of acyclic fibrations and $\mathsf{AC}$ of acyclic cofibrations, and by construction and properties of weak equivalences on $[n]$ we have that $\mathsf{W} = \mathsf{AF} \circ \mathsf{AC}$ satisfies 2-out-of-3. As such, it is enough to show that the pairs $(\mathsf{AC}, \mathsf{AC}^{\boxslash})$ and $({}^{\boxslash}\mathsf{AF},\mathsf{AF})$ are compatible weak factorization systems. In particular we need to show for each pair $(\mathcal{L},\mathcal{R})$ that
\begin{enumerate}
	\item every morphism admits a factorization as $f=pi$ for $i \in \mathcal{L}$ and $p \in \mathcal{R}$,
	\item $\mathcal{L} \boxslash \mathcal{R}$,
	\item $\mathcal{L}$ and $\mathcal{R}$ are closed under retracts,
\end{enumerate}
and additionally the compatibility relation that
\begin{enumerate}
	\item[(4)] $\mathsf{AF} \subseteq \mathsf{AC}^{\boxslash}$ (equivalently that $\mathsf{AC} \subseteq {}^{\boxslash}\mathsf{AF}$).
\end{enumerate}
Condition 2 is immediate by construction, as is Condition 3 as we have no non-trivial retracts in a poset. Thus, it remains to prove Conditions 1 and 4.

We begin with Condition 1. Let $f \colon X \to Y$ in $[n]$. If $Y \leq  R^{\mathrm{max}}(X)$ (where $R^{\mathrm{max}(X)}$ is taken in the contractible submodel in which $X$ lives in), then we are working in a contractible submodel, and the factorization exists by assumption. Therefore, assume that $ R^{\mathrm{max}}(X) < Y$. We will factor $f$ as
\[
X \xrightarrow{\mathmakebox[3em]{\in \mathsf{AC}}} R^{\mathrm{max}}(X) \xrightarrow{\mathmakebox[3em]{\in \mathsf{AC}^{\boxslash}}} Y.
\]
We need to show that the map $R^{\mathrm{max}}(X) \to Y$ is in $\mathsf{AC}^\boxslash$, that is, it admits lifts for all commutative diagrams of the form
\[\xymatrix{
A \ar[r] \ar@{^(->}[d]_{\vsim} & R^{\mathrm{max}}(X) \ar[d] \\
B \ar[r] & Y.
}\]
We assume no such lift exists, that is, there is a morphism $ R^{\mathrm{max}}(X) \to B$. In this case we have $A < R^{\mathrm{max}}(X)  < B$ and as such we are within one of our contractible submodel structures; in particular, by the decomposition property of weak equivalences, the maps $A \to R^{\mathrm{max}}(X) $ and $R^{\mathrm{max}}(X)  \to B$ are weak equivalences. As such, we construct a pushout of the following form
\[\xymatrix{
A \ar[r]^-{\sim} \ar@{^(->}[d]_{\vsim}  & R^{\mathrm{max}} (X) \ar[d]^{\vsim} \\
B \ar@{=}[r] & B. \pullbackcorner
}\]
As pushouts of acyclic cofibrations are acyclic cofibrations (again, noting that the above diagram takes place purely within a contractible submodel), we have that $R^\mathrm{max}(X) \to B$ is an acyclic cofibration which contradicts the maximality of $R^\mathrm{max}(X)$. As such, we have constructed the required factorization. A dual argument takes care of the other factorization using $Q^\mathrm{min}(X)$. This concludes the proof of Condition 1.

We now move to proving Condition 4. We shall show that $\mathsf{AF} \subseteq \mathsf{AC}^{\boxslash}$. Let $f \colon X \to Y$ be in $\mathsf{AF}$. We wish to show that for any acyclic cofibraton $A \to B$ we have a lift
\[\xymatrix{
A \ar[r] \ar@{^(->}[d]_{\vsim}& X \ar@{->>}[d]^{\vsim}  \\
B \ar@{..>}[ur] \ar[r]  & Y. 
}\]
Assume such a lift does not exist. This means that $A \leq  X < B$, which implies by decomposition property that $A \simeq X$ and $X \simeq B$. By 2-out-of-3 we have $B \simeq Y$, and as such the entire diagram lives in a contractible submodel. It follows that $B \cong X$ as we know this diagram must admit a lift in the contractible submodel. This means that we have proved Condition 4.

% We now move to proving Condition 4. We shall show that $\mathsf{AF} \subseteq \mathsf{AC}^{\boxslash}$. Let $f \colon A \to B$ be in $\mathsf{AF}$. We wish to show that for any acylic cofibraton $X \to Y$ we have a lift
% \[\xymatrix{
% X \ar[r] \ar@{^(->}[d]_{\vsim}& A \ar@{->>}[d]^{\vsim}  \\
% Y \ar@{..>}[ur] \ar[r]  & B. 
% }\]
% Assume such a lift does not exist. This means that $X \leq  A < Y$, which implies by decomposition property that $X \simeq A$ and $A \simeq Y$. By 2-out-of-3 we have $Y \simeq B$, and as such the entire diagram lives in a contractible submodel. It follows that $Y \cong A$ as we know this diagram must admit a lift in the contractible submodel. This means that we have proved Condition 4.

Finally, the uniqueness of the model structure is guaranteed by the fact that any model structure is uniquely determined by its acyclic fibrations and acyclic cofibrations (Remark~\ref{rem:overdetermined}). 
\end{proof}

\begin{remark}
Alternatively, one could use the language of transfer systems as introduced in Definition~\ref{def:transfersystem} to provide a more economical proof of Theorem~\ref{thm:main}. Indeed, the selection of contractible submodels is equivalent to giving a disjoint collection of transfer systems on $[n]$. One can check that the disjoint union of transfer systems is indeed a transfer system, in particular, we obtain a single transfer system on $[n]$. We know that this then gives rise to a weak factorization system by Proposition~\ref{fooqwresult}. A similar argument for the left class gives the other weak factorization system. In particular this covers Conditions 1--3 required in the above proof. We have chosen to present the above proof as it is entirely self contained.
\end{remark}

\begin{corollary}\label{cor:main}
There is a bijection of sets
\[
\{\text{model structures on }[n]\} \cong \{\text{choices of contractible submodels on }[n]\}.
\]
\end{corollary}

\begin{proof}
The $\subseteq$ direction is given by Theorem~\ref{thm:main} while the $\supseteq$ direction is Proposition~\ref{prop:modelgivessub}.
\end{proof}

\begin{corollary}\label{cor:transferaf}
Every model structure on $[n]$ determines and is determined by an interval partition and a choice of transfer system on each connected component.\hfill\qedsymbol
\end{corollary}

\begin{remark}
Immediately from Corollary~\ref{cor:main} we can obtain results regarding the homotopy category of any model structure on $[n]$. Suppose $[n]$ is equipped with a model structure with associated interval partition $\mathfrak{p} = [0,a_1] \amalg [a_1+1,a_2] \amalg \cdots \amalg [a_k+1,n]$. Then we claim that $\operatorname{Ho}([n]) \cong [k]$ where the $k$ here corresponds to the $k$ appearing in the final term of the interval partition. Indeed, we know that in a poset each weak equivalence class has a unique bifibrant object, and the interval partition tells us that we have $k$ weak equivalence classes. The result then follows from Lemma~\ref{lem:homotopycategoryposet}.
\end{remark}

\begin{remark}
The proof of Theorem~\ref{thm:main} critically hinges on the shape of the category $[n]$, and the proof cannot be generalised to other lattices, which we now give an explicit example of. Consider the lattice $[1] \times [1]$ and fix a contractible submodel as in Figure~\ref{problematic1x1}. That is, we choose a contractible model structure on the subcategory consisting of the top objects and the morphism between them.

\begin{figure}[h]
\centering
\begin{tikzpicture}
\node (pol) [draw=none, minimum size=2.25cm, regular polygon, regular polygon sides=4] at (0,0) {};
\foreach \n [count=\nu from 0, remember=\n as \lastn, evaluate={\nu+\lastn}] in {1,2,...,4}
\node[anchor=\n*(360/4)-180] at (pol.corner \n){};
\foreach \n in {1,2,...,4}
\draw[fill = black] (pol.corner \n) circle (2pt);
\draw[right hook->, shorten >=3mm, shorten <=3mm](pol.corner 2) -- (pol.corner 1) node[midway,sloped,above]{$\sim$};
\draw[->, shorten >=3mm, shorten <=3mm](pol.corner 2) -- (pol.corner 3);
\draw[->, shorten >=3mm, shorten <=3mm](pol.corner 1) -- (pol.corner 4);
\draw[->, shorten >=3mm, shorten <=3mm](pol.corner 3) -- (pol.corner 4);
\draw[->, shorten >=3mm, shorten <=3mm](pol.corner 2) -- (pol.corner 4);
\end{tikzpicture}
\caption{A selection of contractible submodels on the poset $[1] \times [1]$.}\label{problematic1x1}
\end{figure}
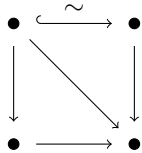

Then we cannot extend this selection of contractible submodels to a model structure. For this to be a model category the bottom morphism must be a fibration for the factorization axiom to hold. Indeed, the bottom morphism must factor as a fibration followed by an acyclic cofibration, and this can only happen when the morphism is a fibration (as all weak equivalences have been selected). However, it cannot be a fibration as it does not admit a lift against the acyclic cofibration due to the shape of the poset.
\end{remark}

\subsection{Properties of model structures on $[n]$}\label{sec:properties}

We finish this section by proving some properties that an arbitrary model structure on $[n]$ possesses. Although none of these properties will be overly surprising to experts given the simple nature of the category $[n]$, it is nonetheless instructive to exhibit their proofs. As a consequence of what we prove in this section, it will follow that one can always take left and right Bousfield localizations. We will revisit this in the final section of the paper.

Recall that a model structure is called \emph{left proper} if the weak equivalences are preserved by pushouts along cofibrations. Dually, it is called \emph{right proper} if the weak equivalences are preserved by pullbacks along fibrations. A model structure which is both left and right proper is said to be \emph{proper}.

\begin{lemma}
Let $[n]$ be equipped with a model structure. Then the weak equivalences of $[n]$ are stable under pushouts and pullbacks along arbitrary maps. As such every model structure on $[n]$ is proper.
\end{lemma}

\begin{proof}
We shall prove stability under pushouts, with the statement regarding pullbacks following from a completely dual argument. We have a pushout diagram
\[\xymatrix{
X \ar[d]_{\vsim} \ar[r] & \ar[d] A \\ Y \ar[r] & B.
}\]
From the construction of the pushout in $[n]$ we have that $B = \max(A,Y)$. If $Y\leq A$ then our diagram is
\[\xymatrix{
X \ar[d]_{\vsim} \ar[r] & \ar@{=}[d] A \\ Y \ar[r] & A
}\]
where we use the fact that an isomorphism is in particular a weak equivalence. In the other case that $A < Y$ our diagram is
\[\xymatrix{
X \ar[d]_{\vsim} \ar[r] & \ar[d] A \\ Y \ar@{=}[r] & Y.
}\]
By the decomposition property of weak equivalences in $[n]$ it follows that in fact all morphisms in this diagram are weak equivalences.
\end{proof}

% \begin{corollary}
% Every model structure on $[n]$ is proper.\hfill\qedsymbol
% \end{corollary}

Another important property of model categories that will be of use to us is \emph{cofibrant generation}. This boils down to asking that there is a suitable set of maps $I$ and $J$ that generate the cofibrations and acyclic cofibrations under suitable constructions. When the underlying category of the model structure is moreover locally presentable, we say that the model structure is \emph{combinatorial}. The following result is clear from the definition.

\begin{lemma}
Every model structure on $[n]$ is combinatorial.
\end{lemma}

% \begin{proof}
% As there is only a finite set of morphisms in $[n]$, it follows that it is cofibrantly generated. The category $[n]$ is locally presentable because it is a complete lattice.
% \end{proof}

As $[n]$ is a category with all products and coproducts, it follows that it admits two monoidal structures:
\begin{itemize}
    \item the cartesian monoidal structure with $X \otimes Y = \min(X,Y)$ and unit $n$;
    \item the cocartesian monoidal structure with $X \otimes Y = \max(X,Y)$ and unit $0$.
\end{itemize}
    
Recall that a model structure on a monoidal category $(\mathcal{C},\otimes, \mathbb{I})$ is \emph{monoidal} if the following hold.
\begin{enumerate}
    \item (Pushout-product axiom) For every pair of cofibrations $f \colon X \to Y$  and $f' \colon X' \to Y'$, their \emph{pushout-product} 
\[
f \square f' \colon (X \otimes Y') \coprod_{X \otimes X'} (Y \otimes X') \to Y \otimes Y'
\]
is also a cofibration. Moreover, it is an acyclic cofibration if either $f$ or $f'$ are.
    \item (Unit axiom) For every cofibrant object $X$ and every cofibrant resolution $\mathbb{I}^{\mathsf{c}} \to \mathbb{I}$ of the monoidal unit, the induced morphism $\mathbb{I}^{\mathsf{c}} \otimes X \to \mathbb{I} \otimes X$ is a weak equivalence.
\end{enumerate}

We will now show that all model structures on the poset $[n]$ interact well with both of its natural monoidal structures.

\begin{lemma}
Let $[n]$ be equipped with a model structure. Then $[n]$ is a monoidal model category with respect to the cartesian monoidal structure. Similarly, $[n]$ is a monoidal model category with respect to the cocartesian monoidal structure.
\end{lemma}

\begin{proof}
The pushout-product axiom takes the form
\[
f \square f' \colon  \max(\min(X,Y'),\min(Y,X')) \to \min(Y,Y').
\]
There are several cases to be checked, but representative are the cases of
\begin{align*}
f \square f' &= f \colon X \to Y,\\
f \square f' &= g \colon X' \to Y,\\
f \square f' &= \mathrm{id}_Y \colon Y \to Y.
\end{align*}
The first case occurs when we have morphisms $X' \to X \to Y \to Y'$. By hypothesis $f$ is a cofibration. If $f'$ is an acyclic cofibration, then by the decomposition property, $f$ is also acyclic as required.

The second scenario occurs when we have morphisms $X \to X' \to Y \to Y'$. The map $g$ is a cofibration whenever $f$ is, as it is possible to construct a lift as follows
\[\xymatrix{
X \ar[d] \ar@/_1pc/[dd]_{f} \ar[dr]  \\ 
X' \ar[r] \ar[d] & A \ar[d] \\
Y \ar[r] \ar@{..>}[ur] & B,
}\]
using that the lift exists for $f \colon X \to Y$. Again this morphisms is moveover acyclic when either $f$ or $f'$ are by using the decomposition property. 

The third case happens when we have morphisms $X \to Y \to X' \to Y'$. Clearly this is always an acyclic cofibration as it is an isomorphism.

For the unit axiom, we note that it is enough to check for the cofibrant replacement $Q^{\mathrm{min}}(n)$. If $\min(Q^{\mathrm{min}}(n),X) = X$ then we are done. If $\min(Q^{\mathrm{min}}(n),X) = Q^{\mathrm{min}}(n)$ then there are morphisms $$Q^{\mathrm{min}}(n) \to X \to n$$ where the composite is a weak equivalence. By the decomposition property we see that $Q^{\mathrm{min}}(n) \to X $ is a weak equivalence as required.

The cocartesian monoidal structure is proved in a similar fashion, so we exclude the proof.
\end{proof}

% \begin{lemma}
% Let $[n]$ be equipped with a model structure. Then $[n]$ is a monoidal model category with respect to the cocartesian monoidal structure.
% \end{lemma}

% \begin{proof}
% We note that as the monoidal unit is $0$ the unit axiom always holds as the initial object is always cofibrant. The pushout-product becomes
% \[
% f \square f' \colon \max \left(\max(X,Y'), \max(Y,X') \right) \to \max(Y,Y')
% \]
% which is the identity morphism on $\max(Y,Y')$, and as such, an (acyclic) cofibration as required.
% \end{proof}

\section{Homotopical combinatorics of \texorpdfstring{$[n]$}{[n]}}\label{sec:enum}

We now undertake the enumeration of model structures on $[n]$ and produce full information regarding the homotopical combinatorics of these lattices. To do so, we first recall \emph{transfer systems} and their relation to contractible model structures in \ref{sec:enumcont}. In~\ref{sec:exampleof2}, we work out the example of $[2]$, which motivates the techniques of~\ref{subsec:enum} where our main theorems are proved. We place these results in the context of premodel structures in~\ref{subsec:premodel}.%, and finally we exhibit a duality on model structures on self-dual lattices (including $[n]$) in~\ref{subsec:dual}.

\subsection{Transfer systems}\label{sec:enumcont}

In the previous section we proved that to understand model structures on $[n]$ it is enough to consider collections of contractible submodels. Therefore, to proceed with a classification of model structures on $[n]$ it is enough to have a classification of contractible model structures on $[k]$ for all $0\leq k \leq n$, which is equivalent to classifying all weak factorization systems on $[k]$. In light of Proposition~\ref{fooqwresult} this is in turn equivalent to classifying all transfer systems on $[n]$.

This latter classification was achieved in~\cite{bbr}. The starting point is the study of $N_\infty$-operads for finite groups $G$ in the sense of Blumberg--Hill~\cite{blumberghill}. These operads capture varying classes of multiplicative norm maps supported by equivariant ring spectra, and encode levels of homotopy commutativity in the equivariant setting.  It was proved in~\cite{bbr} and independently in~\cite{rubin} that these objects of interest are equivalent to the combinatorial data of \emph{$G$-transfer systems} as we now define.

\begin{definition}
Let $G$ be a finite group and let $\operatorname{Sub}(G)$ denote its lattice of subgroups. Then a \emph{$G$-transfer system} is a relation $\mathcal{R}$ on $\operatorname{Sub}(G)$ that refines inclusion and satisfies the following.
\begin{enumerate}
\item $H \, \mathcal{R} \, H$ for all $H \leq G$,
\item $K \, \mathcal{R} \, H$ and $L \, \mathcal{R} \, K$ implies $L \, \mathcal{R} \, H$,
\item $K \, \mathcal{R} \, H$ implies that $(gKg^{-1}) \, \mathcal{R} \, (gHg^{-1})$ for all $g \in G$,
\item $K \, \mathcal{R} \, H$ and $M \leq H$ implies $(K \cap M ) \, \mathcal{R}\,  M$.
\end{enumerate}
\end{definition}

\begin{example}
There are 5 $C_{p^2}$-transfer systems. These can be found in~\cite[Example 14]{bbr}, and we will recall them in the next section in Table~\ref{wfs2}.
\end{example}

One then observes that a $G$-transfer system for $G$ Abelian is equivalent to a transfer system in the sense of Definition~\ref{def:transfersystem} on $\operatorname{Sub}(G)$. In the case that $G=C_{p^n}$, the lattice of subgroups is exactly $[n]$, our lattice of interest. As such, a classification of $C_{p^n}$-$N_\infty$-operads leads directly to a classification of transfer systems, whence contractible model structures, on $[n]$. 

\begin{proposition}[\cite{bbr}]\label{prop:bbr}
Let $G = C_{p^n}$. Then there are $\mathsf{Cat}(n+1)$ many $G$-transfer systems, where $\mathsf{Cat}(n+1)$ is the $(n+1)$-th Catalan number.
\end{proposition}

Combining Proposition~\ref{prop:bbr} and Proposition~\ref{fooqwresult} we arrive at the following classification result for contractible model structures on $[n]$ as required.

\begin{corollary}\label{transferaf}
There are $\mathsf{Cat}(n+1)$ many contractible model structures on $[n]$. The edges of the transfer system correspond to the acyclic fibrations of the model structure.\hfill\qedsymbol
\end{corollary}

\begin{remark}\label{rem:odotoperation}
At this point it is worth discussing the following method, as features of it will appear in Section~\ref{sec:bousfield} when we discuss Bousfield localizations, and is the key input in the proof of Proposition~\ref{prop:bbr}. We will denote by $\operatorname{Tr}([n])$ the collection of all transfer systems on $[n]$. A key ingredient to studying transfer systems is an operation
\[
\odot \colon \operatorname{Tr}([i]) \times \operatorname{Tr}([j]) \to \operatorname{Tr}([i+j+2]).
\]
By taking the outlook of a transfer systems as a certain subgraph of $[n]$, we can pictorially represent this operation as in Figure~\ref{dotdia}.

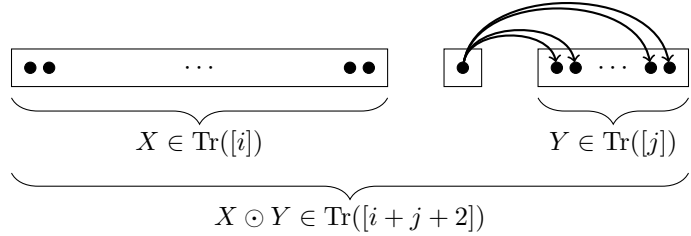
\begin{figure}[h!]
\centering
\begin{tikzpicture}[scale = 1]
\draw [black] (0,0) rectangle (5,0.5);
\draw [black] (5.75,0) rectangle (6.25,0.5);
\draw [black] (7,0) rectangle (9,0.5);

\node [fill,circle,draw,inner sep = 0pt, outer sep = 0pt, minimum size=1.5mm] (from) at (6,.25) {};

\node [fill,circle,draw,inner sep = 0pt, outer sep = 0pt, minimum size=1.5mm] (to1) at (8.75,.25) {};
\node [fill,circle,draw,inner sep = 0pt, outer sep = 0pt, minimum size=1.5mm] (to2) at (8.5,.25) {};

\node [fill,circle,draw,inner sep = 0pt, outer sep = 0pt, minimum size=1.5mm] (to3) at (7.25,.25) {};
\node [fill,circle,draw,inner sep = 0pt, outer sep = 0pt, minimum size=1.5mm] (to4) at (7.5,.25) {};

\node [fill,circle,draw,inner sep = 0pt, outer sep = 0pt, minimum size=1.5mm] at (4.75,.25) {};
\node [fill,circle,draw,inner sep = 0pt, outer sep = 0pt, minimum size=1.5mm] at (4.5,.25) {};

\node [fill,circle,draw,inner sep = 0pt, outer sep = 0pt, minimum size=1.5mm] at (0.25,.25) {};
\node [fill,circle,draw,inner sep = 0pt, outer sep = 0pt, minimum size=1.5mm] at (0.5,.25) {};

\draw[->,bend left=80,thick] (from) to (to3);
\draw[->,bend left=80,thick] (from) to (to4);
\draw[->,bend left=80,thick] (from) to (to1);
\draw[->,bend left=80,thick] (from) to (to2);

\node at (8.025,0.25) {$\cdots$};
\node at (2.525,0.25) {$\cdots$};

\draw [decorate,decoration={brace,amplitude=10pt}]
(5.0,-0.15) -- (0.0,-0.15) node [black,midway,yshift=-0.6cm] { $X \in \operatorname{Tr}([i])$};

\draw [decorate,decoration={brace,amplitude=10pt}]
(9.0,-0.15) -- (7.0,-0.15) node [black,midway,yshift=-0.6cm] { $Y \in \operatorname{Tr}([j])$};

\draw [decorate,decoration={brace,amplitude=10pt}]
(9.0,-1.15) -- (0.0,-1.15) node [black,midway,yshift=-0.6cm] { $X \odot Y \in \operatorname{Tr}([i+j+2])$};

\end{tikzpicture}
\caption{The general picture for the operation $X \odot Y$. We highlight that the last vertex in $X$ occurs at position $i$, the pivot point is in spot $i+1$, and the first vertex of the shift of $Y$ occurs at position $i+2$. Note that we allow either $X$ or $Y$ to be empty.
}\label{dotdia}
\end{figure}

We refer the reader to \cite[Section 3.2]{bbr} for a detailed description of this operation. Essentially, we take the two directed graphs for each transfer system, and glue them together using a \emph{pivot point}. A simple counting exercise shows that this must be a subgraph of $[i+j+2]$. 
\end{remark}

The proof of Proposition~\ref{prop:bbr} follows from identifying the reccurence relation for the Catalan numbers together with the recurrence relation for $\operatorname{Tr}([n])$ coming from the following proposition.

\begin{proposition}[\cite{bbr}]\label{prop:build}
Let $X \in \operatorname{Tr}([i])$ and $Y \in \operatorname{Tr}([j])$. Then $X \odot Y \in  \operatorname{Tr}([i+j+2])$. Moreover, the converse is true, that is, if  $Z \in \operatorname{Tr}([n])$, then there exist unique $i$, $j$, $X \in \operatorname{Tr}([i])$ and $Y \in \operatorname{Tr}([j])$ such that $n=i+j+2$ and $Z=X \odot Y$.
\end{proposition}

\subsection{An example}\label{sec:exampleof2}

Before we get any further into our enumeration efforts, let us illustrate the theory alongside an example. We will study first non-trivial total order that we will consider, namely $[2]$, which we will diagrammatically represent as follows.

\begin{figure}[H]
\centering
\begin{tikzpicture}%%model3
\node (pol) [draw=none, minimum size=2.5cm, regular polygon, regular polygon sides=3] at (0,0) {};
\foreach \n [count=\nu from 0, remember=\n as \lastn, evaluate={\nu+\lastn}] in {1,2,...,3}
\node[anchor=\n*(360/3)-180] at (pol.corner \n){};
\foreach \n in {1,2,...,3}
\draw[fill = black] (pol.corner \n) circle (2pt);
\draw[->, shorten >=3mm, shorten <=3mm](pol.corner 1) -- (pol.corner 2);
\draw[->, shorten >=3mm, shorten <=3mm](pol.corner 1) -- (pol.corner 3);
\draw[->, shorten >=3mm, shorten <=3mm](pol.corner 2) -- (pol.corner 3);
\end{tikzpicture}
\caption{The poset $[2]$. The minimal element 0 is at the top, and then we go counter-clockwise around the diagram.}
\end{figure}
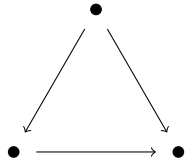

By Corollary \ref{cor:transferaf} we first need to determine the possible weak equivalence patterns on $[2]$. These are $$\mathsf{W}_1=\{ \mbox{isos} \},\,\, \mathsf{W}_2=\{ \mbox{isos}, 0 \rightarrow 1 \},\,\, \mathsf{W}_3=\{ \mbox{isos}, 1 \rightarrow 2 \},\,\, \mathsf{W}_4=\{ \mbox{all} \},$$
which correspond to the four interval paritions
\[
\pp_1 = [0,0] \amalg [1,1] \amalg [2,2], \,\, \, \pp_2 = [0,1] \amalg [2,2], \, \,\, \pp_3 = [0,0] \amalg [1,2], \, \,\, \pp_4 = [0,2].
\]

We now equip each of the interval partitions $\pp_i$ with a selection of contractible submodel structures.
\begin{itemize}
    \item[($\pp_1$)] Here we must chose three contractible model structures, each on the trivial lattice $[0]$. There is only one transfer system on $[0]$, and as such there is a unique choice of contractible model structure with interval partition $\pp_1$. Of course, this is unsurprising as this is exactly the trivial model structure.
    \item[($\pp_2$)] In this situation we need to pick a contractible sub model on a $[0]$ and a $[1]$. There are two choices of transfer system on $[1]$, and as such there are two model structures for the interval partition $\pp_2$.
    \item[($\pp_3$)] This situation is analogous to the one in $\pp_2$, and as such we obtain 2 more model structures here.
    \item[($\pp_4$)] We now need to pick a transfer system on $[2]$, of which there are five. These are  the contractible model structures on $[2]$.
\end{itemize}
  
  Concluding the above discussion, we see that there are
  \[
  1+2+2+5 = 10
  \]
  model structures on $[2]$. Let us summarise the construction that we will use to form explicit descriptions of all of the model structures on $[2]$.

%The above method of forming the model structures is good for enumeration purposes and provides us with the weak equivalences and acyclic fibrations. We know that this is sufficient to recover the other classes of maps uniquely to describe the model structure. However, we now explore a slightly different method, exploiting the relationship between model structures, weak factorization systems and transfer systems, which gives us all classes of maps during the construction.

\begin{construction}\leavevmode
\begin{enumerate}
\item Find all transfer systems on $[2]$. 
\item Construct the weak factorization system with those transfer systems as right sets. 
\item Find all pairs of transfer systems $(\mathcal{L}_1, \mathcal{R}_1)$ and $(\mathcal{L}_2, \mathcal{R}_2)$ with $\mathcal{R}_1 \subseteq \mathcal{R}_2$. 
\item Those pairs now give us premodel structures with 
\[
(\mathcal{L}_1, \mathcal{R}_1)= (\mathsf{C}, \mathsf{AF}) \,\,\,\mbox{and}\,\,\,
(\mathcal{L}_2, \mathcal{R}_2)= (\mathsf{AC}, \mathsf{F}),
\]
with weak equivalences given by $\mathsf{W}=\mathsf{AF}\circ\mathsf{AC}.$
\item Finally, identify those premodel structures where $\mathsf{W}$ satisfies 2-out-of-3. These are all possible model structures on $[2]$.
\end{enumerate}
\end{construction}

\begin{remark}
We note that this recipe would work on any other finite poset, but as the examples of $[n]$ show, doing so by hand will quickly become unmanageable, as would the enumeration.
\end{remark}

The five transfer systems and their corresponding left class, denoted as $(\mathcal{L},\mathcal{R})$ on $[2]$, are listed in Table~\ref{wfs2}. The circled numbers are only decoration so that we may refer to specific right classes.

\begin{table}[H]
\begin{tabular}{|c|c|}
\hline
\rowcolor[HTML]{C0C0C0} 
$\mathcal{L}$ & $\mathcal{R}$ \\ \hline

\begin{tikzpicture}
\node (pol) [draw=none, minimum size=1.75cm, regular polygon, regular polygon sides=3] at (0,0) {};
\foreach \n [count=\nu from 0, remember=\n as \lastn, evaluate={\nu+\lastn}] in {1,2,...,3}
\node[anchor=\n*(360/3)-180] at (pol.corner \n){};
\foreach \n in {1,2,...,3}
\draw[fill = black] (pol.corner \n) circle (2pt);
\draw[->, shorten >=3mm, shorten <=3mm](pol.corner 1) -- (pol.corner 2);
\draw[->, shorten >=3mm, shorten <=3mm](pol.corner 1) -- (pol.corner 3);
\draw[->, shorten >=3mm, shorten <=3mm](pol.corner 2) -- (pol.corner 3);
\end{tikzpicture}

&

\begin{tikzpicture}
\node (pol) [draw=none, minimum size=1.75cm, regular polygon, regular polygon sides=3] at (0,0) {};
\foreach \n [count=\nu from 0, remember=\n as \lastn, evaluate={\nu+\lastn}] in {1,2,...,3}
\node[anchor=\n*(360/3)-180] at (pol.corner \n){};
\foreach \n in {1,2,...,3}
\draw[fill = black] (pol.corner \n) circle (2pt);
\node at (0,0) {\circled{1}};
\end{tikzpicture}

\\ \hline

\begin{tikzpicture}
\node (pol) [draw=none, minimum size=1.75cm, regular polygon, regular polygon sides=3] at (0,0) {};
\foreach \n [count=\nu from 0, remember=\n as \lastn, evaluate={\nu+\lastn}] in {1,2,...,3}
\node[anchor=\n*(360/3)-180] at (pol.corner \n){};
\foreach \n in {1,2,...,3}
\draw[fill = black] (pol.corner \n) circle (2pt);
\draw[->, shorten >=3mm, shorten <=3mm](pol.corner 1) -- (pol.corner 3);
\draw[->, shorten >=3mm, shorten <=3mm](pol.corner 2) -- (pol.corner 3);
\end{tikzpicture}

&

\begin{tikzpicture}
\node (pol) [draw=none, minimum size=1.75cm, regular polygon, regular polygon sides=3] at (0,0) {};
\foreach \n [count=\nu from 0, remember=\n as \lastn, evaluate={\nu+\lastn}] in {1,2,...,3}
\node[anchor=\n*(360/3)-180] at (pol.corner \n){};
\foreach \n in {1,2,...,3}
\draw[fill = black] (pol.corner \n) circle (2pt);
\draw[->, shorten >=3mm, shorten <=3mm](pol.corner 1) -- (pol.corner 2);
\node at (0,0) {\circled{2}};
\end{tikzpicture}

\\ \hline

\begin{tikzpicture}
\node (pol) [draw=none, minimum size=1.75cm, regular polygon, regular polygon sides=3] at (0,0) {};
\foreach \n [count=\nu from 0, remember=\n as \lastn, evaluate={\nu+\lastn}] in {1,2,...,3}
\node[anchor=\n*(360/3)-180] at (pol.corner \n){};
\foreach \n in {1,2,...,3}
\draw[fill = black] (pol.corner \n) circle (2pt);
\draw[->, shorten >=3mm, shorten <=3mm](pol.corner 1) -- (pol.corner 2);
\end{tikzpicture}

&

\begin{tikzpicture}
\node (pol) [draw=none, minimum size=1.75cm, regular polygon, regular polygon sides=3] at (0,0) {};
\foreach \n [count=\nu from 0, remember=\n as \lastn, evaluate={\nu+\lastn}] in {1,2,...,3}
\node[anchor=\n*(360/3)-180] at (pol.corner \n){};
\foreach \n in {1,2,...,3}
\draw[fill = black] (pol.corner \n) circle (2pt);
\draw[->, shorten >=3mm, shorten <=3mm](pol.corner 2) -- (pol.corner 3);
\node at (0,0) {\circled{3}};
\end{tikzpicture}

\\ \hline

\begin{tikzpicture}
\node (pol) [draw=none, minimum size=1.75cm, regular polygon, regular polygon sides=3] at (0,0) {};
\foreach \n [count=\nu from 0, remember=\n as \lastn, evaluate={\nu+\lastn}] in {1,2,...,3}
\node[anchor=\n*(360/3)-180] at (pol.corner \n){};
\foreach \n in {1,2,...,3}
\draw[fill = black] (pol.corner \n) circle (2pt);
\draw[->, shorten >=3mm, shorten <=3mm](pol.corner 2) -- (pol.corner 3);
\end{tikzpicture}

&

\begin{tikzpicture}
\node (pol) [draw=none, minimum size=1.75cm, regular polygon, regular polygon sides=3] at (0,0) {};
\foreach \n [count=\nu from 0, remember=\n as \lastn, evaluate={\nu+\lastn}] in {1,2,...,3}
\node[anchor=\n*(360/3)-180] at (pol.corner \n){};
\foreach \n in {1,2,...,3}
\draw[fill = black] (pol.corner \n) circle (2pt);
\draw[->, shorten >=3mm, shorten <=3mm](pol.corner 1) -- (pol.corner 2);
\draw[->, shorten >=3mm, shorten <=3mm](pol.corner 1) -- (pol.corner 3);
\node at (0,0) {\circled{4}};
\end{tikzpicture}

\\ \hline

\begin{tikzpicture}
\node (pol) [draw=none, minimum size=1.75cm, regular polygon, regular polygon sides=3] at (0,0) {};
\foreach \n [count=\nu from 0, remember=\n as \lastn, evaluate={\nu+\lastn}] in {1,2,...,3}
\node[anchor=\n*(360/3)-180] at (pol.corner \n){};
\foreach \n in {1,2,...,3}
\draw[fill = black] (pol.corner \n) circle (2pt);
\end{tikzpicture}

&

\begin{tikzpicture}
\node (pol) [draw=none, minimum size=1.75cm, regular polygon, regular polygon sides=3] at (0,0) {};
\foreach \n [count=\nu from 0, remember=\n as \lastn, evaluate={\nu+\lastn}] in {1,2,...,3}
\node[anchor=\n*(360/3)-180] at (pol.corner \n){};
\foreach \n in {1,2,...,3}
\draw[fill = black] (pol.corner \n) circle (2pt);
\draw[->, shorten >=3mm, shorten <=3mm](pol.corner 1) -- (pol.corner 2);
\draw[->, shorten >=3mm, shorten <=3mm](pol.corner 1) -- (pol.corner 3);
\draw[->, shorten >=3mm, shorten <=3mm](pol.corner 2) -- (pol.corner 3);
\node at (0,0) {\circled{5}};
\end{tikzpicture}

\\ \hline
\end{tabular}
\vspace{3mm}
\caption{All possible weak factorization systems on the poset $[2]$.}\label{wfs2}
\end{table}

From Table~\ref{wfs2} we can now read off all possible premodel structures by considering those pairs $(\mathcal{L}_1, \mathcal{R}_1)$, $(\mathcal{L}_2, \mathcal{R}_2)$ such that $\mathcal{R}_1 \subseteq \mathcal{R}_2$. Specifically, the lattice of inclusions for the right class takes the form depicted in Figure \ref{associahedra}.

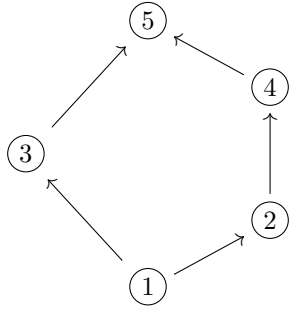
\begin{figure}[H]
    \centering \[
    \xymatrix@C=1.0em@R=0.5em{& & \circled{5} \\
    &&&& \circled{4} \ar[ull] \\
    \circled{3} \ar[uurr] \\
    &&&& \circled{2} \ar[uu] \\
    &&\circled{1} \ar[urr] \ar[uull]
    }
    \]
    \caption{The lattice of inclusions for transfer systems on $[2]$.}
    \label{associahedra}
\end{figure}

One then reads off all possible elements in the interval lattice of Figure~\ref{associahedra}, of which there are 13. We list the resulting premodel structures in Table~\ref{premodel2}, where the numbering corresponds to the interval we are considering.

\begin{table}[H]
\begin{tabular}{c|c|c|c|c}
\hline
\multicolumn{1}{|c|}{

\begin{tikzpicture}%%1
\node (pol) [draw=none, minimum size=2.5cm, regular polygon, regular polygon sides=3] at (0,0) {};
\foreach \n [count=\nu from 0, remember=\n as \lastn, evaluate={\nu+\lastn}] in {1,2,...,3}
\node[anchor=\n*(360/3)-180] at (pol.corner \n){};
\foreach \n in {1,2,...,3}
\draw[fill = black] (pol.corner \n) circle (2pt);

\draw[right hook->, shorten >=3mm, shorten <=3mm](pol.corner 1) -- (pol.corner 2) node[midway,sloped,above]{$\sim$};

\draw[right hook->, shorten >=3mm, shorten <=3mm](pol.corner 1) -- (pol.corner 3) node[midway,sloped,above]{$\sim$};

\draw[right hook->, shorten >=3mm, shorten <=3mm](pol.corner 2) -- (pol.corner 3) node[midway,sloped,below]{$\sim$};
\node at (0,-0.1) {\circled{1}$\leqslant$\circled{1}};
\end{tikzpicture}

} & 

\begin{tikzpicture}%%model2
\node (pol) [draw=none, minimum size=2.5cm, regular polygon, regular polygon sides=3] at (0,0) {};
\foreach \n [count=\nu from 0, remember=\n as \lastn, evaluate={\nu+\lastn}] in {1,2,...,3}
\node[anchor=\n*(360/3)-180] at (pol.corner \n){};
\foreach \n in {1,2,...,3}
\draw[fill = black] (pol.corner \n) circle (2pt);

\draw[right hook->>, shorten >=3mm, shorten <=3mm](pol.corner 1) -- (pol.corner 2);

\draw[right hook->>, shorten >=3mm, shorten <=3mm](pol.corner 1) -- (pol.corner 3);

\draw[right hook->, shorten >=3mm, shorten <=3mm](pol.corner 2) -- (pol.corner 3) node[midway,sloped,below]{$\sim$};
\node at (0,-0.1) {\circled{1}$\leqslant$\circled{4}};
\end{tikzpicture}

& 

\begin{tikzpicture}%%model3
\node (pol) [draw=none, minimum size=2.5cm, regular polygon, regular polygon sides=3] at (0,0) {};
\foreach \n [count=\nu from 0, remember=\n as \lastn, evaluate={\nu+\lastn}] in {1,2,...,3}
\node[anchor=\n*(360/3)-180] at (pol.corner \n){};
\foreach \n in {1,2,...,3}
\draw[fill = black] (pol.corner \n) circle (2pt);

\draw[right hook->>, shorten >=3mm, shorten <=3mm](pol.corner 1) -- (pol.corner 2);

\draw[right hook->>, shorten >=3mm, shorten <=3mm](pol.corner 1) -- (pol.corner 3);

\draw[right hook->>, shorten >=3mm, shorten <=3mm](pol.corner 2) -- (pol.corner 3);
\node at (0,-0.1) {\circled{1}$\leqslant$\circled{5}};
\end{tikzpicture}

& 

\begin{tikzpicture}%%model4
\node (pol) [draw=none, minimum size=2.5cm, regular polygon, regular polygon sides=3] at (0,0) {};
\foreach \n [count=\nu from 0, remember=\n as \lastn, evaluate={\nu+\lastn}] in {1,2,...,3}
\node[anchor=\n*(360/3)-180] at (pol.corner \n){};
\foreach \n in {1,2,...,3}
\draw[fill = black] (pol.corner \n) circle (2pt);

\draw[right hook->, shorten >=3mm, shorten <=3mm](pol.corner 1) -- (pol.corner 2) node[midway,sloped,above]{$\sim$};

\draw[right hook->, shorten >=3mm, shorten <=3mm](pol.corner 1) -- (pol.corner 3);

\draw[right hook->>, shorten >=3mm, shorten <=3mm](pol.corner 2) -- (pol.corner 3);
\node at (0,-0.1) {\circled{1}$\leqslant$\circled{3}};

\end{tikzpicture}

& \multicolumn{1}{c|}{

\begin{tikzpicture}%%model5
\node (pol) [draw=none, minimum size=2.5cm, regular polygon, regular polygon sides=3] at (0,0) {};
\foreach \n [count=\nu from 0, remember=\n as \lastn, evaluate={\nu+\lastn}] in {1,2,...,3}
\node[anchor=\n*(360/3)-180] at (pol.corner \n){};
\foreach \n in {1,2,...,3}
\draw[fill = black] (pol.corner \n) circle (2pt);

\draw[->>, shorten >=3mm, shorten <=3mm](pol.corner 1) -- (pol.corner 2) node[midway,sloped,above]{$\sim$};

\draw[right hook->, shorten >=3mm, shorten <=3mm](pol.corner 1) -- (pol.corner 3) node[midway,sloped,above]{$\sim$};

\draw[right hook->, shorten >=3mm, shorten <=3mm](pol.corner 2) -- (pol.corner 3) node[midway,sloped,below]{$\sim$};
\node at (0,-0.1) {\circled{2}$\leqslant$\circled{2}};

\end{tikzpicture}

}  \\ \hline
\multicolumn{1}{|c|}{

\begin{tikzpicture}%%model6
\node (pol) [draw=none, minimum size=2.5cm, regular polygon, regular polygon sides=3] at (0,0) {};
\foreach \n [count=\nu from 0, remember=\n as \lastn, evaluate={\nu+\lastn}] in {1,2,...,3}
\node[anchor=\n*(360/3)-180] at (pol.corner \n){};
\foreach \n in {1,2,...,3}
\draw[fill = black] (pol.corner \n) circle (2pt);

\draw[->>, shorten >=3mm, shorten <=3mm](pol.corner 1) -- (pol.corner 2) node[midway,sloped,above]{$\sim$};

\draw[right hook->>, shorten >=3mm, shorten <=3mm](pol.corner 1) -- (pol.corner 3);

\draw[right hook->>, shorten >=3mm, shorten <=3mm](pol.corner 2) -- (pol.corner 3);
\node at (0,-0.1) {\circled{2}$\leqslant$\circled{5}};

\end{tikzpicture}

} & 

\begin{tikzpicture}%%Model7
\node (pol) [draw=none, minimum size=2.5cm, regular polygon, regular polygon sides=3] at (0,0) {};
\foreach \n [count=\nu from 0, remember=\n as \lastn, evaluate={\nu+\lastn}] in {1,2,...,3}
\node[anchor=\n*(360/3)-180] at (pol.corner \n){};
\foreach \n in {1,2,...,3}
\draw[fill = black] (pol.corner \n) circle (2pt);

\draw[->>, shorten >=3mm, shorten <=3mm](pol.corner 1) -- (pol.corner 2) node[midway,sloped,above]{$\sim$};

\draw[->>, shorten >=3mm, shorten <=3mm](pol.corner 1) -- (pol.corner 3) node[midway,sloped,above]{$\sim$};

\draw[right hook->, shorten >=3mm, shorten <=3mm](pol.corner 2) -- (pol.corner 3) node[midway,sloped,below]{$\sim$};
\node at (0,-0.1) {\circled{4}$\leqslant$\circled{4}};

\end{tikzpicture}

&  

\begin{tikzpicture}%%model8
\node (pol) [draw=none, minimum size=2.5cm, regular polygon, regular polygon sides=3] at (0,0) {};
\foreach \n [count=\nu from 0, remember=\n as \lastn, evaluate={\nu+\lastn}] in {1,2,...,3}
\node[anchor=\n*(360/3)-180] at (pol.corner \n){};
\foreach \n in {1,2,...,3}
\draw[fill = black] (pol.corner \n) circle (2pt);

\draw[->>, shorten >=3mm, shorten <=3mm](pol.corner 1) -- (pol.corner 2) node[midway,sloped,above]{$\sim$};

\draw[->>, shorten >=3mm, shorten <=3mm](pol.corner 1) -- (pol.corner 3) node[midway,sloped,above]{$\sim$};

\draw[->>, shorten >=3mm, shorten <=3mm](pol.corner 2) -- (pol.corner 3) node[midway,sloped,below]{$\sim$};
\node at (0,-0.1) {\circled{5}$\leqslant$\circled{5}};
\end{tikzpicture}

&  

\begin{tikzpicture}%%model9
\node (pol) [draw=none, minimum size=2.5cm, regular polygon, regular polygon sides=3] at (0,0) {};
\foreach \n [count=\nu from 0, remember=\n as \lastn, evaluate={\nu+\lastn}] in {1,2,...,3}
\node[anchor=\n*(360/3)-180] at (pol.corner \n){};
\foreach \n in {1,2,...,3}
\draw[fill = black] (pol.corner \n) circle (2pt);

\draw[right hook->>, shorten >=3mm, shorten <=3mm](pol.corner 1) -- (pol.corner 2);

\draw[->>, shorten >=3mm, shorten <=3mm](pol.corner 1) -- (pol.corner 3);

\draw[->>, shorten >=3mm, shorten <=3mm](pol.corner 2) -- (pol.corner 3) node[midway,sloped,below]{$\sim$};
\node at (0,-0.1) {\circled{3}$\leqslant$\circled{5}};

\end{tikzpicture}

& \multicolumn{1}{c|}{

\begin{tikzpicture}
\node (pol) [draw=none, minimum size=2.5cm, regular polygon, regular polygon sides=3] at (0,0) {};
\foreach \n [count=\nu from 0, remember=\n as \lastn, evaluate={\nu+\lastn}] in {1,2,...,3}
\node[anchor=\n*(360/3)-180] at (pol.corner \n){};
\foreach \n in {1,2,...,3}
\draw[fill = black] (pol.corner \n) circle (2pt);

\draw[right hook->, shorten >=3mm, shorten <=3mm](pol.corner 1) -- (pol.corner 2) node[midway,sloped,above]{$\sim$};

\draw[->, shorten >=3mm, shorten <=3mm](pol.corner 1) -- (pol.corner 3) node[midway,sloped,above]{$\sim$};

\draw[->>, shorten >=3mm, shorten <=3mm](pol.corner 2) -- (pol.corner 3) node[midway,sloped,below]{$\sim$};
\node at (0,-0.1) {\circled{3}$\leqslant$\circled{3}};
\end{tikzpicture}

} \\ \hline
                        & 
                        \cellcolor[HTML]{FFCCC9} 
\begin{tikzpicture}
\node (pol) [draw=none, minimum size=2.5cm, regular polygon, regular polygon sides=3] at (0,0) {};
\foreach \n [count=\nu from 0, remember=\n as \lastn, evaluate={\nu+\lastn}] in {1,2,...,3}
\node[anchor=\n*(360/3)-180] at (pol.corner \n){};
\foreach \n in {1,2,...,3}
\draw[fill = black] (pol.corner \n) circle (2pt);

\draw[right hook->>, shorten >=3mm, shorten <=3mm](pol.corner 1) -- (pol.corner 2);
\draw[right hook->, shorten >=3mm, shorten <=3mm](pol.corner 1) -- (pol.corner 3) node[midway,sloped,above]{$\sim$};
\draw[right hook->, shorten >=3mm, shorten <=3mm](pol.corner 2) -- (pol.corner 3) node[midway,sloped,below]{$\sim$};
\node at (0,-0.1) {\circled{1}$\leqslant$\circled{2}};

\end{tikzpicture}
                        
                        & 
                         \cellcolor[HTML]{FFCCC9} 
\begin{tikzpicture}
\node (pol) [draw=none, minimum size=2.5cm, regular polygon, regular polygon sides=3] at (0,0) {};
\foreach \n [count=\nu from 0, remember=\n as \lastn, evaluate={\nu+\lastn}] in {1,2,...,3}
\node[anchor=\n*(360/3)-180] at (pol.corner \n){};
\foreach \n in {1,2,...,3}
\draw[fill = black] (pol.corner \n) circle (2pt);

\draw[->>, shorten >=3mm, shorten <=3mm](pol.corner 1) -- (pol.corner 2) node[midway,sloped,above]{$\sim$};

\draw[right hook->>, shorten >=3mm, shorten <=3mm](pol.corner 1) -- (pol.corner 3);

\draw[right hook->, shorten >=3mm, shorten <=3mm](pol.corner 2) -- (pol.corner 3) node[midway,sloped,below]{$\sim$};
\node at (0,-0.1) {\circled{2}$\leqslant$\circled{4}};

\end{tikzpicture}
                        
                        & 
                         \cellcolor[HTML]{FFCCC9} 
\begin{tikzpicture}
\node (pol) [draw=none, minimum size=2.5cm, regular polygon, regular polygon sides=3] at (0,0) {};
\foreach \n [count=\nu from 0, remember=\n as \lastn, evaluate={\nu+\lastn}] in {1,2,...,3}
\node[anchor=\n*(360/3)-180] at (pol.corner \n){};
\foreach \n in {1,2,...,3}
\draw[fill = black] (pol.corner \n) circle (2pt);

\draw[->>, shorten >=3mm, shorten <=3mm](pol.corner 1) -- (pol.corner 2) node[midway,sloped,above]{$\sim$};
\draw[->>, shorten >=3mm, shorten <=3mm](pol.corner 1) -- (pol.corner 3) node[midway,sloped,above]{$\sim$};
\draw[  right hook->>, shorten >=3mm, shorten <=3mm](pol.corner 2) -- (pol.corner 3);
\node at (0,-0.1) {\circled{4}$\leqslant$\circled{5}};

\end{tikzpicture}
                    
                        &                         \\ \cline{2-4}
\end{tabular}
\vspace{3mm}
\caption{The 13 possible premodel structures on the poset $[2]$. The arrows use the decoration introduced in Definition~\ref{defn:modelcat}. The shaded bottom cells represent premodel structures which are not model structures.}\label{premodel2}
\end{table}

Of the 13 premodel structures appearing in Table \ref{premodel2}, we see that the three in the bottom row do not satisfy the 2-out-of-3 property for weak equivalences (that is, they do not correspond to interval partitions of $[2]$). As such, we (re)conclude that there are 10 model structures on the lattice $[2]$. %Note that this time, however, we have constructed all classes of maps $\mathsf{W}, \mathsf{F}$ and $\mathsf{C}$ without needing to appeal to further lifting arguments.

\subsection{Enumerating model structures}\label{subsec:enum}

We will now state and prove the enumeration result for $Q([n])$. Recall from Corollary~\ref{transferaf} that we already have an explicit enumeration of the size of $W([n])$, the collection of weak factorization systems of the poset $[n]$. We can now use this in conjunction with Corollary~\ref{cor:main} to prove the promised enumeration result. First, we require an intermediary lemma. We also note that this lemma can be extracted from \cite{shapiro}, but we provide an independent proof that is of interest in creating an explicit bijection between $Q([n])$ and the collection of monotonic functions from $[n]$ to itself. See Remark~\ref{rmk:bijection} for further discussion of this.

\begin{lemma}[{\cite[Propositions 3.1 and 3.3]{shapiro}}]\label{lemma:CpId}
For every positive integer $n$,
\[
  \binom{2n-1}{n} = \sum_{k=1}^n ~ \sum_{i_1+\cdots+i_k = n} ~\prod_{j=1}^k \mathsf{Cat}(i_j)
\]
where the second sum runs over ordered $k$-tuples of positive integers summing to $n$.
\end{lemma}

\begin{proof}
Consider the lattice paths on the grid $[0,n]\times [0,n]$ that begin with a step from $(0,0)$ to $(0,1)$, end at $(n,n)$, and only take steps up or to the right by one unit; by choosing which steps go to the right, we see that there are precisely $\binom{2n-1}{n}$ of these. Each such path determines an ordered partition of $n$ according to when it \emph{crosses} (not just touches) the diagonal.

Now consider all paths that produce a particular ordered partition
\[
  i_1+\cdots+i_k=n.
\]
Such paths consist of a $2i_1$-step path (weakly) above the diagonal, followed by a $2i_2$-step path (weakly) below the diagonal, \emph{etc}., with the total number of possibilities counted by
\[
  \prod_{j=1}^k \Cat(i_j).
\]
(See Figure \ref{fig:CpId} for a graphical depiction. Our count uses the well-known fact that there are $\Cat(k)$ Dyck paths of length $2k$.) Adding up these terms over all ordered partitions of $n$ gives the identity.
\end{proof}

\begin{figure}
\begin{center}
\begin{tikzpicture}
\draw[step=1cm,gray,very thin] (0,0) grid (7,7);
\draw[dotted,thick] (0,0) -- (3,3);
\draw[dashed,thick] (3,3) -- (5,5);
\draw[loosely dotted,thick] (5,5) -- (7,7);
\draw[thick] (0,0) -- (0,2) -- (1,2) -- (1,3) -- (4,3) -- (4,4) -- (5,4) -- (5,7) -- (7,7);
\end{tikzpicture}
\end{center}
\caption{An east-north lattice path from $(0,0)$ to $(7,7)$ beginning with a step from $(0,0)$ to $(0,1)$, as in the first proof of Lemma \ref{lemma:CpId}. The diagonal changes styles when the path crosses it, yielding the ordered partition $3+2+2$ of $7$. There are a total of $\Cat(3)\Cat(2)\Cat(2)$ such paths corresponding to this partition.}\label{fig:CpId}
\end{figure}
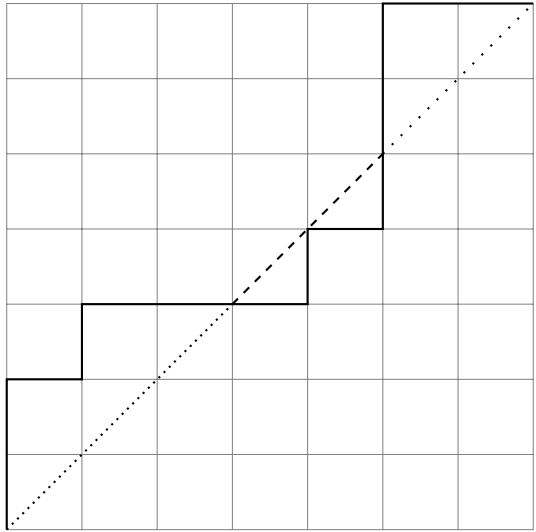

%Let us start with recalling Corollary~\ref{transferaf}.

% \begin{corollary}\label{cor:wfsn}
% Let $W([n])$ be the collection of weak factorization systems of the poset $[n]$. Then
% \[
% \# W([n]) = \frac{1}{n+2}\binom{2n+2}{n+1} = \mathsf{Cat}(n+1),
% \]
% the $(n+1)$-th Catalan number.\hfill\qedsymbol
% \end{corollary}

% Now that we have recalled the classification of contractible model structures on the poset $[n]$, we can use this in conjunction with Corollary~\ref{cor:main} to prove the promised enumeration result.

We are now in a position to prove the main theorem.

\begin{theorem}\label{thm:Qn}
Let $Q([n])$ be the collection of model structures on the poset $[n]$. Then
\[
\# Q([n]) = \binom{2n+1}{n}.
\]
\end{theorem}
\begin{proof}
By Corollary~\ref{cor:main}, there is a bijective correspondence between $Q([n])$ and choices of contractible submodels on $[n]$. We may partition the latter objects according to interval partitions, and then assign a contractible model structure to each block. Given an interval partition $\mathfrak{p} = [0,a_1] \amalg [a_1+1,a_2] \amalg \cdots \amalg [a_k+1,n]$ of $[n]$, set $a_0 = -1$, $a_{k+1} = n$, and define
\[
  C_{\mathfrak{p}} := \prod_{i=0}^{k+1} \mathsf{Cat}(a_{i+1}-a_i).
\]
By Corollaries~\ref{cor:transferaf} and \ref{transferaf}, there are $C_{\mathfrak{p}}$ many contractible submodels on $[n]$ with  components $\mathfrak{p}$.  Thus the total number of contractible submodels on $[n]$ is
\[
  \sum_{\mathfrak{p}} C_{\mathfrak{p}}
\]
where the sum runs over interval partitions of $[n]$. It remains to show that this sum of products of Catalan numbers equals $\binom{2n+1}{n}$, which follows from Lemma~\ref{lemma:CpId} using Remark \ref{rmk:partition}, which tells us such partitions are in bijective correspondence with ordered partitions of the integer $n+1$.
\end{proof}

\begin{example}
Let $n=4$, then there are are $\binom{9}{4} = 126$ model structures on $[4]$. These model structures are given in Figure~\ref{fig:n4} at the end of this paper.
\end{example}

\begin{remark}\label{rmk:bijection}
%Using the proof of Lemma~\ref{lemma:CpId}, 

We can describe an explicit bijection
\[
  \Phi\colon Q([n])\longrightarrow \End([n])
\]
where $\End([n])$ denotes the set of monotonic functions $[n]\to [n]$. We represent the latter with staircase paths on the grid $[n+1] \times [n+1]$ including the edge $(0,0)\rightarrow (0,1)$ ending at $(n,n)$. Beginning with a model structure $M$ on $[n]$, record its contractible submodels $M_0,\ldots,M_k$.
We now know that each of the contractible submodels is determined by a transfer system. The Catalan enumeration result of Remark~\ref{rem:odotoperation} implies that there is a bijection between transfer systems and Dyck paths, i.e. staircase paths strictly above the diagonal. 
   Let $D_i$ denote the Dyck path corresponding to the transfer system determining $M_i$, and let
$\overline{D}$ denote the reflection of a Dyck path over the diagonal. Then the concatenation of paths
\[
  P = \overline{D_0}D_1\overline{D_2}D_3\cdots
\]
is an east-north lattice path from $(0,0)$ to $(n+1,n+1)$ that begins with a step $(0,0)$ to $(0,1)$.  Define a function $f\colon [0,n]\to [1,n+1]$ so that $f(k)$ is the highest point on $P$ in the column $\{k\}\times [0,n+1]$. We finally set
\[
  \Phi(M) := f-1\in \End([n]).
\]

For example, if $M$ is a model structure on $[6]$ inducing the path $P$ of Figure~\ref{fig:CpId}, then the associated endomorphism of $[6]$ takes the values
\[
  1,~2,~2,~2,~3,~6,~6
\]
in order.
\end{remark}

By restricting our attention to lattice paths that cross the diagonal exactly $k$ times, we can produce refined statistics on $Q([n])$ counting model structures with a particular homotopy category $[k]$.

\begin{theorem}\label{thm:refined}
There are
\[
\frac{2(k+1)}{n+k+2}\binom{2n+1}{n-k}
\]
model structures on $[n]$ whose homotopy category is isomorphic to $[k]$.%\hfill\qedsymbol
\end{theorem}
\begin{proof}
By the proofs of Theorem~\ref{thm:Qn} and Lemma~\ref{lemma:CpId} and the bijection of Remark~\ref{rmk:bijection}, we see that a model structure on $[n]$ with homotopy category $[k]$ corresponds to an east-north lattice path on $[0,n+1]\times [0,n+1]$ that begins with a step from $(0,0)$ to $(0,1)$ and ends at $(n+1,n+1)$ which crosses the diagonal \emph{exactly} $k$ times. There are precisely
\[
  \sum_{i_1+\cdots+i_{k+1}=n+1}\prod_{j=1}^{k+1}\Cat(i_j)
\]
such paths, where the sum is over length $k+1$ ordered partitions of $n+1$.
By \cite[Propositions 2.1 and 3.3]{shapiro}, we know that this sum is precisely
\[
 \frac{2(k+1)}{n+k+2}\binom{2n+1}{n-k}
\]
as desired.
\end{proof}

\begin{example}
Table~\ref{table:shapiro} gives the number of model structures on $[n]$ with homotopy category $[k]$ for $0\le k\le n \le 5$. It can be found elsewhere in combinatorial literature and is (up to an indexing shift) sometimes called \emph{Shapiro's Catalan triangle} after  \cite{shapiro}.
\begin{table}[H]
\begin{tabular}{c|cccccccc}
\diagbox{$n$}{$k$} & 0  & 1  & 2  & 3 & 4 & 5 &  & Total \\ \cline{1-7} \cline{9-9} 
0               & 1  &    &    &   &   &  && 1     \\
1               & 2  & 1  &    &   &   &  && 3     \\
2               & 5  & 4  & 1  &   &   &  && 10    \\
3               & 14 & 14 & 6  & 1 &   &  && 35    \\
4               & 42 & 48 & 27 & 8 & 1 &  && 126   \\
5               & 132 & 165 & 110 & 44 & 10 & 1 && 462
\end{tabular}
\caption{Shapiro's Catalan triangle, indexed so that the $(k,n)$ entry corresponds to the number of model structures on $[n]$ with homotopy category isomorphic to $[k]$. The first column displays the (shifted) Catalan numbers, corresponding to contractible model structures. The final column displays the row sum, \emph{i.e.}, the total number of model structures on $[n]$.}\label{table:shapiro}
\end{table}
\end{example}

\begin{remark}\label{rmk:shapiro}
Shapiro arrives at the results of \cite{shapiro} by considering pairs of non-intersecting east-north lattice paths. The \emph{distance} between such paths ending at $(a,b)$ and $(c,d)$ is defined to be $|a-c|$. It turns out that the number of model structures on $[n]$ with homotopy category $[k]$ is equal to the number of pairs of non-intersecting east-north lattice paths of length $n+1$ and distance $k+1$. We leave it as an open question whether there is a natural bijection between these structures.
\end{remark}

\begin{remark}\label{rmk:allfib}
To conclude this subsection, we mention that there are two additional classes of model structures which we can count. First, consider those model structures on a lattice $\mathcal P$ for which all morphisms are fibrations. The rest of the model structure is then uniquely specified by a choice of acyclic fibrations satisfying the $2$-out-of-$3$ property. These are given by so-called \emph{saturated} transfer systems, which are studied in \cite{hmoo}. A consequence of the main theorem of that paper is that when $\mathcal P = [n]$, there are exactly $2^n$ such model structures with $\mathsf{F} = \mathsf{All}$. (In fact, a saturated transfer system on $[n]$ is completely determined by the covering relations in the system. Since there are $n$ covering relations in $[n]$, this explains the count $2^n$.)

By the standard duality on model structures, it is also the case that there are $2^n$ model structures on $[n]$ with $\mathsf{C} = \mathsf{All}$.
\end{remark}

\subsection{Relation to premodel structures}\label{subsec:premodel}

In Definition~\ref{def:premodel}, we defined model categories to be premodel categories such that the weak equivalences satisfy the 2-out-of-3 property. In Section~\ref{sec:exampleof2} we used this description and saw that we had 13 premodel structures, of which 10 were model structures.

In the general case of $[n]$, we shall now see that it is possible to always enumerate the number of premodel structures. To do so, we need to know all compatible pairs $(\mathcal{L}_1,\mathcal{R}_1)$, $(\mathcal{L}_2,\mathcal{R}_2)$ of weak factorization systems on $[n]$ with $\mathcal{R}_1 \subseteq \mathcal{R}_2$. The first observation that we will need is that we fully understand the lattice structure of $W([n])$, the set of weak factorization systems on $[n]$. 

We recall that the $(n+1)$-Tamari lattice (also sometimes called the $(n+1)$-associahedron) is the classical lattice one obtains for example, by  considering binary trees with $n+1$ leaves, ordered by tree rotation operations~\cite{stasheff}. We will denote by $\mathcal{A}(n+1)$, the $(n+1)$-Tamari lattice.

\begin{proposition}[\cite{bbr}]\label{prop:asso}
The poset $W([n])$ under inclusion of the right class is isomorphic to $\mathcal{A}(n+1)$.
\end{proposition}

\begin{example}
In the case of $n=2$ we obtain the 3-Tamari lattice as already displayed in Figure~\ref{associahedra}.
\end{example}

As such, we need only understand the cardinality of the lattice of intervals of the Tamari lattice  to obtain a count of premodel structures on $([n])$. Such an analysis has been achieved by Chapoton.

\begin{proposition}[\cite{chapoton}]
Let $I_n$ be the set of intervals of the $n$-Tamari lattice. Then
\[
\#I_n = \frac{2}{n(n+1)}\binom{4n+1}{n-1}.
\]
\end{proposition}

By combining the two above propositions (and noting the index shift required), we arrive at the following corollary.

\begin{corollary}\label{cor:premodel}
Let $P([n])$ be the collection of (unique) premodel structures on the poset $[n]$. Then
\[
\# P([n]) = \frac{2}{(n+1)(n+2)} \binom{4n+5}{n}.
\]
\hfill\qedsymbol
\end{corollary}

As such, we have closed formulas for both the number of premodel structures and the numbers of model structures. We now record an observation about the relation between these enumerations.

\begin{proposition}\label{prop:asym}
Quillen model structures and premodel structures on $[n]$ satisfy the asymptotic relationship
\[
\frac{\# Q([n])}{\# P([n])} \sim c2^{dn}n^2
\]
where
\begin{align*}
c &= \frac{243 \sqrt{3/2}}{1024} \approx 0.290638,\\[5pt]
d &= 3 \log_2(3) - 6 \approx -1.24511.
\end{align*}
In particular, we have
\[
\lim_{n \to \infty }\frac{\#{Q}([n])}{\#{P}([n])} = 0.
\]
\end{proposition}

\begin{proof}
This follows from Stirling's approximation for $n!$ (see, for instance, \cite[\S 1.2.11]{aocp}), Theorem \ref{thm:Qn}, and Corollary \ref{cor:premodel}.

\end{proof}

\section{Bousfield localizations}\label{sec:bousfield}

Bousfield localizations are a way of formally adding weak equivalences to an existing model structure. Left Bousfield localizations retain the original cofibrations while adding weak equivalences, while right localizations retain the original fibrations. We will show in this section that every model structure on $[n]$ can be created from the trivial model structure via a zig-zag of left and right Bousfield localizations.

Let us first recap some basic definitions. These are very general, but we will soon see that they simplify considerably when the underlying category is $[n]$.  Before we begin, we need to point out that the definitions make use of homotopy mapping spaces. A mapping space $\Map(X,Y)$ is a simplicial set with $\pi_0(\Map(X,Y))=[X,Y].$ We refer the reader to \cite[Chapter 5.2]{hovey} for more details. Fortunately, when the underlying category is $[n]$, the mapping spaces become very simple indeed.

\begin{lemma}\label{lem:mappingspace}
Let $[n]$ be equipped with a model structure. Then for any two objects $X,Y \in [n]$ we have
\[
\Map(X,Y)=
\begin{cases}
\{\ast\}& \text{if } X \simeq Y,\\
\{\ast\} & \text{if } X \not\simeq Y,\,\,X < Y,\\
\, \,\varnothing & \text{if } X \not\simeq Y, X > Y.
\end{cases}
\]
\end{lemma}

\begin{proof}
Firstly, we recall that the $n$-simplices of the simplicial set $\Map(X,Y)$ are given by
\[
\Map(X,Y)_n=\operatorname{Hom}(X^{\mathsf{c}} \otimes \Delta^n, Y^{\mathsf{f}})
\]
where $A \otimes K$ is the tensor of the object $A$ with a simplicial set $K$.
To form this tensor, we replace a cofibrant object $A$ with the constant cosimplicial object with $A$ in every degree, and then use the fact that the category of cosimplicial objects in a model category has a tensor with simplicial sets. In particular, for any cosimplicial object $A^\bullet$, we have that $A^\bullet \otimes \Delta^n = A^n.$ 
One can compute that as simplicial sets,

%\marginpar{Angelica: the referee had an issue with us saying that $X^c$ and $Y^f$ were just one point, and I see why. Is this sentence necessary given the explanation? Also, The explanation seems to say that the mapping simplicial set is always constant, is that true? Or am I misunderstanding something?}
\[
\Map(X,Y) \simeq \operatorname{Hom}(X^{\mathsf{c}}, Y^{\mathsf{f}}).
\]
In particular, we see that 
\[
\Map(X,Y) \simeq \operatorname{Hom}(X^{\mathsf{cf}}, Y^{\mathsf{cf}}).
\]
If $X$ and $Y$ are weakly equivalent, they have the same (unique) bifibrant replacement by Lemma \ref{lem:uniquebif}. If they are not, then the bifibrant replacement of $X$ will be smaller than the bifibrant replacement of $Y$ if and only if $X <Y$. This proves the claim. 
\end{proof}

\begin{remark}
Note that although a mapping space for a model structure on $[n]$ is not a particularly complicated object, it does come with some hidden subtleties. It is possible to have $\operatorname{Hom}(X,Y) = \varnothing$ but $\Map(X,Y) = \{ \ast\}$ (although not vice versa). For example, consider the category $[1]$ equipped with the model structure where the single non-identity map is an acyclic fibration. Then $\operatorname{Hom}(1,0) = \varnothing$ but $\Map(1,0) = \{\ast\}$ as both objects have the same bifibrant replacement (namely 0).
\end{remark}

We now introduce the general definitions required to discuss Bousfield localization. We will then go further and show how these definitions can be explicitly described in the case where the base category is $[n]$.

\begin{definition}
Let $\C$ be a model category and $\mathcal{W}$ a class of morphisms in $\C$. 
\begin{itemize}
    \item A fibrant object $Z \in \C$ is \emph{$\mathcal{W}$-local} if for all $f: X \longrightarrow Y \in \mathcal{W}$ the map
\[
f^*: \Map(Y,Z) \longrightarrow \Map(X,Z)
\]
is a weak equivalence of simplicial sets. 

\item A morphism $g:A \longrightarrow B$ is a \emph{$\mathcal{W}$-equivalence} if 
\[
g^*: \Map(B,Z) \longrightarrow \Map(A,Z)
\]
is a weak equivalence of simplicial sets for all $\mathcal{W}$-local objects $Z$.
\end{itemize}
\end{definition}

The definition of $\mathcal{W}$-equivalences implies that all elements of $\mathcal{W}$ are $\mathcal{W}$-equivalences, as are all weak equivalences of $\mathcal{C}$. In general, the class of $\mathcal{W}$-equivalences is much larger than $\mathcal{W}$ itself. However, in our case the answer is much simpler for many of the properties.

\begin{lemma}\label{lem:localproperties}
For any model category, the $\mathcal{W}$-equivalences satisfy 2-out-of-3. On $[n]$, the $\mathcal{W}$-equivalences are decomposable. 
\end{lemma}

\begin{proof}
The first point follows from the fact that the 2-out-of-3 property holds between the mapping spaces involved. 

As for the second point, let us assume that $Z$ is $\mathcal{W}$-local, $A \longrightarrow C$ is a $\mathcal{W}$-equivalence and $A \leq B \leq C$. Thus, both $\Map(A,Z)$ and $\Map(C,Z)$ are either empty or a point. 

If they are both empty, then $A>Z$, $C>Z$, $A\not\simeq Z$ and $C \not\simeq Z.$ Thus, we also have to have $B>Z$. Furthermore, $B$ cannot be weakly equivalent to $Z$, because otherwise decomposability of weak equivalences would imply that $A$ is also weakly equivalent to $Z$, which would be a contradiction. We therefore have $\Map(B,Z)=\varnothing$, as required.

If $\Map(A,Z)$ and $\Map(C,Z)$ are both equal to a point, then we have the following cases. 
If $B \simeq Z,$ then $\Map(B,Z)$ is also a point.
If $B \not\simeq Z$, then we have to show that $B<Z.$ This again breaks up into the following cases.
\begin{itemize}
\item $Z \simeq C:$ This means that $B$ must be smaller than both $C$ and $Z$, otherwise decomposability of weak equivalences would imply that $B \simeq Z.$
\item $Z \not\simeq C:$ This means that $C<Z,$ and therefore $B<Z$ too, as required.
\end{itemize}

\end{proof}

As the name suggests, right Bousfield localizations, also called colocalizations (or cellularizations), are a dual concept to left Bousfield localizations. As they are not as commonly used as left localizations, we will recall the definitions and just state the technical lemmas without proof.

\begin{definition}Let $\mathcal{W}$ be a class of morphisms in a model category $\C$. 
\begin{itemize}
\item A cofibrant object $Z \in \C$ is \emph{$\mathcal{W}$-colocal} if
\[
f_*: \Map(Z,X) \longrightarrow \Map(Z,Y)
\]
is a weak equivalence for all $f:X \longrightarrow Y$ in $\mathcal{W}$.

\item A morphism $g: A \longrightarrow B$ is a $\mathcal{W}$-coequivalence, if
\[
g_*: \Map(Z,A) \longrightarrow \Map(Z,B)
\]
is a weak equivalence for all $\mathcal{W}$-colocal $Z \in \C.$
\end{itemize}
\end{definition}

Again, we go from a class of morphisms $\mathcal{W}$ to colocal objects and then to $\mathcal{W}$-coequivalences. The $\mathcal{W}$-coequivalences contain both $\mathcal{W}$ and the weak equivalences of $\C$. 

\bigskip
In the case of $\C$ having underlying category $[n]$ we get the following analog of Lemma \ref{lem:localproperties}.

\begin{lemma}\label{lem:colocalproperties}
For any model category, the $\mathcal{W}$-coequivalences satisfy 2-out-of-3. On $[n]$, the $\mathcal{W}$-coequivalences are decomposable.  \hfill\qedsymbol
\end{lemma}

The main purpose of both left and right localizations is to add weak equivalences to an existing model structure. The following result provides an existence result for these localizations.

%n the general case, the existence proof for those new model structures involves a plethora of set theory. However, if the underlying category is just $[n]$ then these assumptions become trivial, and we can use the following result, see \emph{e.g.}, . %In particular, one needs to know that the model structure is combinatorial and proper, but we have proved this holds for any model structure on $[n]$ in Section~\ref{sec:properties}.

\begin{proposition}[{\cite[Chapter 3.3]{hirschhorn}}] Let $\C$ be a combinatorial and proper model category, and let $\mathcal{W}$ be a set of morphisms in $\C$. 
\begin{itemize}
\item There is a model structure $L_\mathcal{W}\C$ on $\C$ such that 
\begin{itemize}
\item the weak equivalences of $L_\mathcal{W}\C$ are the $\mathcal{W}$-equivalences,
\item the cofibrations of $L_\mathcal{W}\C$ are the same as the cofibrations of $\C$.
\end{itemize}
We call $L_\mathcal{W}\C$ the \emph{left Bousfield localization of $\C$ with respect to $\mathcal{W}$}.
\item There is a model structure $R_\mathcal{W}\C$ on $\C$ such that 
\begin{itemize}
\item the weak equivalences of $R_\mathcal{W}\C$ are the $\mathcal{W}$-coequivalences,
\item the fibrations of $R_\mathcal{W}\C$ are the same as the fibrations of $\C$.
\end{itemize}
We call $R_\mathcal{W}\C$ the \emph{right Bousfield localization of $\C$ with respect to $\mathcal{W}$}.
\end{itemize}
\end{proposition}

We see from the definitions of $L_\mathcal{W}\C$ and $R_\mathcal{W}\C$ that the identity gives a left Quillen functor
$\C \longrightarrow L_\mathcal{W}\C$
and a right Quillen functor
$\C \longrightarrow R_\mathcal{W}\C.$ Furthermore, the fibrant replacement in $L_\mathcal{W}\C$ gives a $\mathcal{W}$-equivalence between any object $X$ and a $\mathcal{W}$-local object $L_\mathcal{W}X$. Dually, the cofibrant replacement in $R_\mathcal{W}\C$ gives a $\mathcal{W}$-coequivalence between a colocal object $R_\mathcal{W}X$ and any object in $X$ in $\C$.

Left localization also satisfies the following universal property: if $F:\C \longrightarrow \mathcal{D}$ is a left Quillen functor sending $\mathcal{W}$ to weak equivalences in $\mathcal{D}$, then $F$ factors as a left Quillen functor $F: L_\mathcal{W}\C \longrightarrow \mathcal{D}.$ This implies the following.

\begin{corollary}\label{cor:wsmallleft}
Let $[n]$ be equipped with a model structure. Then the $\mathcal{W}$-equivalences and the $\mathcal{W}$-coequivalences are the smallest subcategory of $[n]$ satisfying decomposition and containing $\mathcal{W}$ and the weak equivalences of $[n]$. \hfill\qedsymbol
\end{corollary}

Our goal in this section is to show that every model structure on $[n]$ can be obtained from the trivial one via a sequence of left and right Bousfield localizations. 

We know from Corollary \ref{cor:transferaf} that a model structure on $[n]$ is uniquely determined by a selection of contractible submodels, that is, by an interval partition and the choice of a transfer system on each component of the partition.

As the elements of the transfer system provide the acyclic fibrations of the model structure, to obtain our main result of this section we therefore start by considering the effect of left and right localizations on the acyclic fibrations. In particular, we shall explore what happens to the acyclic fibrations when left and right localizing at the single map $i\rightarrow i+1$. We denote the corresponding localization and colocalization by
\[
L_i = L_{\{i \longrightarrow i+1\}} \,\,\,\,\mbox{resp.,}\,\,\,\,R_i=R_{\{i \longrightarrow i+1\}}.
\]
By Corollary \ref{cor:wsmallleft} we know that the only new weak equivalence resulting from the above is indeed $i \rightarrow i+1$ and any required composites. In particular we have the following.

\begin{corollary}\label{cor:intpartonbous}
Let $\C$ be a model structure on $[n]$ with corresponding interval partition $$\pp = [0,a_1] \amalg [a_1+1,a_2] \amalg \cdots \amalg [a_k+1,n].$$ Suppose $0 \leq i < n.$ Then the interval partition $\pp'$ for the model structures $L_i\C$ and $R_i\C$ are as follows.
\begin{itemize}
    \item If $i \longrightarrow i+1$ is already a weak equivalence then $\pp' = \pp$.
    \item If $i \longrightarrow i+1$ is not a weak equivalence then in $\pp$ we can find $\cdots \amalg [m,i] \amalg [i+1,j] \amalg \cdots$ (where it is possible that $m=i$ or $i+1 = j$).  Then $\pp'$ is obtained from $\pp$ by replacing this block with $[m,j]$. \hfill\qedsymbol
\end{itemize}
\end{corollary}

% \begin{remark}
% In particular, if the map $i \longrightarrow i+1$ is already a weak equivalence then neither of the localizations $L_i$ or $R_i$ change the model structure. Indeed, the weak equivalences have not changed and both localizations fix one other class of maps. As a model structure is uniquely determined by its weak equivalences and fibrations, resp., cofibrations, the claimed result follows. As such, from now on we will always assume that the map $i \longrightarrow i+1$ is not a weak equivalence.
% \end{remark}

\begin{proposition}\label{lem:newaf}
Let $\C$ be a model structure on $[n]$ with corresponding interval partition $\pp = [0,a_1] \amalg [a_1+1,a_2] \amalg \cdots \amalg [a_k+1,n]$. Suppose $0 \leq i < n$ and that $i \longrightarrow i+1$ is not a weak equivalence. Then the model structures $L_i\C$ and $R_i\C$ are characterised as follows.
\begin{itemize}
\item The weak equivalences for $L_i \C$ are described by the interval partition $\pp'$ as above. The acyclic fibrations of $L_i\C$ are the same as the acyclic fibrations of $\C$.
\item The weak equivalences for $R_i \C$ are described by the interval partition $\pp'$ as above. The acyclic fibrations of $R_i\C$ are the acyclic fibrations of $\C$ with the addition of the arrows $m' \rightarrow j'$, where
\[
i < j' \leq j,
\]
and where $m' \rightarrow i$ are acyclic fibrations in the old model structure.
In particular, all the arrows of the form $i \rightarrow j'$, $i< j' \leq j$ are new acyclic fibrations.

\end{itemize}
\end{proposition}

\begin{proof}
We start with considering $L_i$. The claim about the interval partition is the subject of Corollary~\ref{cor:intpartonbous}. As a left Bousfield localization does not change the cofibrations, it also does not change the acyclic fibrations (which are determined via lifting from the cofibrations). One observes that the we therefore require the transfer system on the block $[m,j]$ to be the disjoint union of the transfer systems on $[m,i]$ and $[i+1,j]$. This is indeed a transfer system as required.

Now let us look at $R_i$. With the previous notation, we would like to determine the acyclic fibrations on the block $[m ,j]$. By definition, any acyclic fibration between $m$ and $i$ and between $i+1$ and $j$ after right localization must have also been an acyclic fibration before right localization.

Let us begin with finding the largest $\ell \ge i$ such that $i \rightarrow \ell$ is a new acyclic fibration.

We will show that $\ell=j$ by proving that $i \rightarrow j$ was a fibration in the old model structure before right localization (and thus is an acyclic fibration after right localization). 
We show that the map $i \rightarrow j$ has the right lifting property with respect to all acyclic cofibrations in the old model structure. By writing out the lifting square, this translates to there being no acyclic cofibration $a \rightarrow b$ such that
\[
a \leq i < b \leq j.
\]
And indeed there is not --- if there were, then $a$ and $b$ would be in different blocks of weak equivalences before right localization. But as $a \rightarrow b$ is in particular a weak equivalence, this would be a contradiction. 
Thus we can conclude that $i \rightarrow j$ is a fibration in the old model structure and therefore an acyclic fibration in $R_i\C$.
This means that $R_i$ adds the edges $i \rightarrow j'$ for all $j'$ between $i$ and $j$ to the acyclic fibrations of $\C$.

In addition to this, any acyclic fibrations $m' \rightarrow i$ will create some new fibrations via composing with the $i \rightarrow j'$ and then applying restrictions. Conversely, note that if $m'\to j'$ is an acyclic fibration after localization, then so is $m'\to i$ by restriction, and hence $m'\to i$ must have been an acyclic fibration before localization. Therefore, all in all, the new acyclic fibrations are precisely the maps $m' \rightarrow j'$, where $i+1 \leq j' \leq j$, and $m' \rightarrow i$ is an acyclic fibration in the old model structure. 
\end{proof}

Recall the operation $\odot$ on transfer systems as discussed in Remark~\ref{rem:odotoperation} which allows us to inductively build any transfer system. The key point is that the Bousfield localizations $L_i$ and $R_i$ allow us to model this operation. The following lemma can be seen from comparing the description in Proposition~\ref{lem:newaf} with the description of $\odot$ in \cite{bbr}. 

\begin{lemma}\label{lem:itisodot}
Let $\C$ be a model structure on $[n]$. Assume the corresponding interval partition contains 
\[\cdots \amalg [m,i-1] \amalg [i,i] \amalg [i+1,j] \amalg \cdots. \]
Then, using the notation $\mathsf{tr}([a,b])$ for the corresponding transfer system on $[a,b]$, we have the following.
\begin{itemize}
    \item $L_{i-1} \C$ models $\mathsf{tr}([m,i-1]) \odot \varnothing$ to obtain a transfer system on $[m,i]$.
    \item $R_{i} \C$ models $\varnothing \odot \mathsf{tr}([i+1,j])$ to obtain a transfer system on $[i,j]$.
    \item $L_{i-1}R_{i}\C$ models $\mathsf{tr}([m,i-1]) \odot \mathsf{tr}([i+1,j])$ to obtain a transfer system on $[m,j]$.\hfill\qedsymbol
\end{itemize}

\end{lemma}

We now have all of the ingredients to state and prove the main result of this section.

\begin{theorem}\label{thm:bousfieldlocn}
Every model structure on $[n]$ can be obtained by a sequence of left and right Bousfield localizations of the form $L_i$ and $R_i$.
\end{theorem}

\begin{proof}

We will argue by induction on the maximal size of the blocks in weak equivalences. If that size is 1, then we have the trivial model structure. Let us assume that we can obtain any transfer system on any choice of contractible submodels with a fixed maximal block size. If we would like to create a transfer system on a bigger block, then this transfer system can be obtained via the $\odot$ operation from Lemma~\ref{lem:itisodot}, which allows us to iteratively build larger transfer systems from smaller ones (see Proposition~\ref{prop:build}). 
\end{proof}

\begin{example}
Let us return to our usual example of $[2]$. Below are two composites of a left and right localization, which, in particular, display the fact that one cannot in general commute left and right localizations.
\[
\begin{gathered}
\begin{tikzpicture}%%model3
\node (pol) [draw=none, minimum size=2.5cm, regular polygon, regular polygon sides=3] at (0,0) {};
\foreach \n [count=\nu from 0, remember=\n as \lastn, evaluate={\nu+\lastn}] in {1,2,...,3}
\node[anchor=\n*(360/3)-180] at (pol.corner \n){};
\foreach \n in {1,2,...,3}
\draw[fill = black] (pol.corner \n) circle (2pt);

\draw[right hook->>, shorten >=3mm, shorten <=3mm](pol.corner 1) -- (pol.corner 2);

\draw[right hook->>, shorten >=3mm, shorten <=3mm](pol.corner 1) -- (pol.corner 3);

\draw[right hook->>, shorten >=3mm, shorten <=3mm](pol.corner 2) -- (pol.corner 3);
\end{tikzpicture}
\end{gathered}
\xrightarrow[\hspace{5em}]{L_1}
\begin{gathered}
\begin{tikzpicture}%%model2
\node (pol) [draw=none, minimum size=2.5cm, regular polygon, regular polygon sides=3] at (0,0) {};
\foreach \n [count=\nu from 0, remember=\n as \lastn, evaluate={\nu+\lastn}] in {1,2,...,3}
\node[anchor=\n*(360/3)-180] at (pol.corner \n){};
\foreach \n in {1,2,...,3}
\draw[fill = black] (pol.corner \n) circle (2pt);

\draw[right hook->>, shorten >=3mm, shorten <=3mm](pol.corner 1) -- (pol.corner 2);

\draw[right hook->>, shorten >=3mm, shorten <=3mm](pol.corner 1) -- (pol.corner 3);

\draw[right hook->, shorten >=3mm, shorten <=3mm](pol.corner 2) -- (pol.corner 3) node[midway,sloped,below]{$\sim$};
\end{tikzpicture}
\end{gathered}
\xrightarrow[\hspace{5em}]{R_0}
\begin{gathered}
\begin{tikzpicture}%%Model7
\node (pol) [draw=none, minimum size=2.5cm, regular polygon, regular polygon sides=3] at (0,0) {};
\foreach \n [count=\nu from 0, remember=\n as \lastn, evaluate={\nu+\lastn}] in {1,2,...,3}
\node[anchor=\n*(360/3)-180] at (pol.corner \n){};
\foreach \n in {1,2,...,3}
\draw[fill = black] (pol.corner \n) circle (2pt);

\draw[->>, shorten >=3mm, shorten <=3mm](pol.corner 1) -- (pol.corner 2) node[midway,sloped,above]{$\sim$};

\draw[->>, shorten >=3mm, shorten <=3mm](pol.corner 1) -- (pol.corner 3) node[midway,sloped,above]{$\sim$};

\draw[right hook->, shorten >=3mm, shorten <=3mm](pol.corner 2) -- (pol.corner 3) node[midway,sloped,below]{$\sim$};
\end{tikzpicture}
\end{gathered}
\]

\[
\begin{gathered}
\begin{tikzpicture}%%model3
\node (pol) [draw=none, minimum size=2.5cm, regular polygon, regular polygon sides=3] at (0,0) {};
\foreach \n [count=\nu from 0, remember=\n as \lastn, evaluate={\nu+\lastn}] in {1,2,...,3}
\node[anchor=\n*(360/3)-180] at (pol.corner \n){};
\foreach \n in {1,2,...,3}
\draw[fill = black] (pol.corner \n) circle (2pt);

\draw[right hook->>, shorten >=3mm, shorten <=3mm](pol.corner 1) -- (pol.corner 2);

\draw[right hook->>, shorten >=3mm, shorten <=3mm](pol.corner 1) -- (pol.corner 3);

\draw[right hook->>, shorten >=3mm, shorten <=3mm](pol.corner 2) -- (pol.corner 3);
\end{tikzpicture}
\end{gathered}
\xrightarrow[\hspace{5em}]{R_0}
\begin{gathered}
\begin{tikzpicture}%%model6
\node (pol) [draw=none, minimum size=2.5cm, regular polygon, regular polygon sides=3] at (0,0) {};
\foreach \n [count=\nu from 0, remember=\n as \lastn, evaluate={\nu+\lastn}] in {1,2,...,3}
\node[anchor=\n*(360/3)-180] at (pol.corner \n){};
\foreach \n in {1,2,...,3}
\draw[fill = black] (pol.corner \n) circle (2pt);

\draw[->>, shorten >=3mm, shorten <=3mm](pol.corner 1) -- (pol.corner 2) node[midway,sloped,above]{$\sim$};

\draw[right hook->>, shorten >=3mm, shorten <=3mm](pol.corner 1) -- (pol.corner 3);

\draw[right hook->>, shorten >=3mm, shorten <=3mm](pol.corner 2) -- (pol.corner 3);
\end{tikzpicture}
\end{gathered}
\xrightarrow[\hspace{5em}]{L_1}
\begin{gathered}
\begin{tikzpicture}%%model5
\node (pol) [draw=none, minimum size=2.5cm, regular polygon, regular polygon sides=3] at (0,0) {};
\foreach \n [count=\nu from 0, remember=\n as \lastn, evaluate={\nu+\lastn}] in {1,2,...,3}
\node[anchor=\n*(360/3)-180] at (pol.corner \n){};
\foreach \n in {1,2,...,3}
\draw[fill = black] (pol.corner \n) circle (2pt);

\draw[->>, shorten >=3mm, shorten <=3mm](pol.corner 1) -- (pol.corner 2) node[midway,sloped,above]{$\sim$};

\draw[right hook->, shorten >=3mm, shorten <=3mm](pol.corner 1) -- (pol.corner 3) node[midway,sloped,above]{$\sim$};

\draw[right hook->, shorten >=3mm, shorten <=3mm](pol.corner 2) -- (pol.corner 3) node[midway,sloped,below]{$\sim$};
\end{tikzpicture}
\end{gathered}
\]
Figure~\ref{2quillen} at the end of this paper is a diagram of all 10 of the model structures on $[2]$ and the left Quillen functors between them. From this one can identify the required left and right Bousfield localizations starting from the trivial model structure to acquire the remaining model structures.
\end{example}

\begin{remark}
We have described a way of adding in single weak equivalences at a time using left and right localizations. One may hope that a single left Bousfield localization followed by a single right Bousfield localization (or the other way around) may do the trick. 

One can verify that there are $2 \cdot 3^n$ possible ways of obtaining model structures on $[n]$ using only two localizations. One then observes that $2 \cdot 3^n < \binom{2n+1}{n}$ for $n > 5$, so by a cardinality argument such a construction cannot work. As such we are forced into a position where we have to use zig-zags of localizations.

For a particular instance of this phenomenon consider the following model structure on $[3]$.

\[
\begin{tikzpicture}[scale=1.25]%%model3
\node (pol) [draw=none, minimum size=3.5cm, regular polygon, regular polygon sides=4] at (0,0) {};
\foreach \n [count=\nu from 0, remember=\n as \lastn, evaluate={\nu+\lastn}] in {1,2,...,4}
\node[anchor=\n*(360/4)-180] at (pol.corner \n){};
\foreach \n in {1,2,...,4}
\draw[fill = black] (pol.corner \n) circle (2pt);

\draw[->>, shorten >=3mm, shorten <=3mm](pol.corner 1) -- (pol.corner 2) node[midway,sloped,above]{$\sim$};

\draw[->>, shorten >=3mm, shorten <=3mm](pol.corner 1) -- (pol.corner 3) node[near start,sloped,above]{$\sim$};

\draw[right hook->, shorten >=3mm, shorten <=3mm](pol.corner 1) -- (pol.corner 4) node[midway,sloped,above]{$\sim$};

\draw[right hook->, shorten >=3mm, shorten <=3mm](pol.corner 2) -- (pol.corner 3) node[midway,sloped,below]{$\sim$};

\draw[shorten >=3mm, shorten <=3mm,line width=8pt, white](pol.corner 2) -- (pol.corner 4) ;

\draw[right hook->, shorten >=3mm, shorten <=3mm](pol.corner 2) -- (pol.corner 4) node[near start,sloped,above]{$\sim$};

\draw[right hook->, shorten >=3mm, shorten <=3mm](pol.corner 3) -- (pol.corner 4) node[midway,sloped,below]{$\sim$};
\end{tikzpicture}
\]
This model structure can be obtained via the composite of localizations $L_2R_0L_1$ applied to the trivial model structure. It is instructive to verify that this model structure cannot be obtained as $L_iR_j$ or $R_jL_i$ for any $i,j$.
\end{remark}

\section{Further directions}\label{sec:further}
We conclude with a list of further directions and questions that would advance our understanding of homotopical combinatorics.
\begin{enumerate}
\item Address Remark~\ref{rmk:shapiro} regarding model structures on $[n]$ and pairs of non-intersecting east-north lattice paths. Separately or in parallel, provide a conceptual explanation for the recurrence relation
\[
  Q_{n,k} = Q_{n-1,k-1}+2Q_{n-1,k}+Q_{n-1,k+1}
\]
where $Q_{n,k}$ is the number of model structures on $[n]$ with homotopy category isomorhic to $[k]$.
\item We have seen that premodel structures $P([n])$ are in bijection with the interval poset of the Tamari lattice, but that $Q([n])$ is only a tiny piece of $P([n])$. Is there a natural way to identify the Tamari intervals corresponding to model structures on $[n]$? What structure does $Q([n])$ have as a subposet of $P([n])$?
\item Relatedly, in \cite{intervals}, a bijection between Tamari intervals and  \emph{triangulations} (rooted planar maps in which all faces have three vertices) of a fixed size is given. Is there a direct bijection between $P([n])$ and triangulations which identifies $Q([n])$ with a special class of triangulations?
\item In Remark~\ref{rmk:bijection} we described a bijection between $Q([n])$ and $\End([n])$. It would be instructive to understand this bijection further, and to see what structures are preserved. For example, $\End([n])$ is a monoid under composition; does this monoid structure translate to a natural monoid structure on $Q([n])$?
\item Following Remark~\ref{rmk:allfib}, we can interpret the results of \cite{hmoo} as an enumeration of model structures on $[m]\times [n]$ for which all morphisms are fibrations. Indeed, there are
\[
  \sum_{j=2}^{m+2}(-1)^{m-j}\begin{Bmatrix}m+1\\j-1\end{Bmatrix}\frac{j!}{2}j^n
\]
such model structures, where $\begin{Bmatrix}k\\\ell\end{Bmatrix}$ is the Stirling number of the second kind counting $\ell$ block partitions of a set with cardinality $k$. What is the full structure of $Q([m]\times [n])$?
\item For what other lattices $\mathcal P$ can one enumerate or determine structural properties of $Q(\mathcal{P})$?
\item Extend the results of Section~\ref{sec:bousfield} to other lattices. For which lattices are all model structures connected by a zig-zag of left- and right-Bousfield localizations?

\item By the construction at \cite[\href{https://kerodon.net/tag/02MC}{Tag 02MC}]{kerodon}, every $(\infty,1)$-category is the localization of a poset category $\mathcal{P}$ at a set of morphisms $\mathsf{W}$. By \cite[Theorem B]{dz21}, we have a criteria under when such a pair $(\mathcal{P},\mathsf{W})$ extends to a model structure. It would be of interest to determine structural and enumerative properties of the collection of model structures on lattices (or on a particular lattice) which present a given locally presentable $(\infty,1)$-category. 

%\kcomment{I haven't thought much about this last question, but am curious about it. Is it substantive? It might be related to Barwick--Kan's work on relative categories as a model for $\infty$-categories. They show that relative posets are the cofibrant objects in their model structure on relative categories, so cofibrant replacement gives a model for $\mathcal {C}$ as $P[W^{-1}]$ (up to equivalence) for some poset $P$ and collection of morphisms $W$. I don't know (a) under what hypotheses we can extend $(P,W)$ to a model structure, and (b) when the equivalence is a Quillen equivalence. But it seems tantalizing!}
\end{enumerate}

%\scomment{Is the extension of $(P,W)$ to a model structure not exactly what Droz--Zakharevich do?}
%\kcomment{Good point. Given a model category $\mathcal C$, there is a constructive model for $(P,W)$ given at \url{https://kerodon.net/tag/02MC}. So the question becomes when does that pair satisfy the conditions of Theorem B in \cite{dz21}. Here's a potential rewrite of (7):

%By the construction at [Kerodon ref], every $(\infty,1)$-category is the localization of a poset category $\mathcal P$ at a set of morphisms $\mathsf W$. Theorem B of \cite{dz21}, provides criteria under which such a pair $(\mathcal P, \mathsf W)$ extends to a model structure. It would be of interest to determine structural and enumerative properties of the collection of model structures on lattices (or on a particular lattice) which present a given locally presentable $(\infty,1)$-category.}

%\kcomment{Add a question about $\End([n])$ and potential monoid structures on $Q([n])$?}

 \bibliography{quillen}\bibliographystyle{alpha}

\newcommand{\etalchar}[1]{$^{#1}$}
\begin{thebibliography}{HMOO22}

\bibitem[BAC]{tobyomar}
T.~Barthel and O.~Antol\'{i}n-Camarena.
\newblock The nine model category structures on the category of sets.
\newblock \url{https://www.matem.unam.mx/~omar/notes/modelcatsets.html}.

\bibitem[Bal21]{handbook}
S.~Balchin.
\newblock {\em A Handbook of Model Categories}, volume~27 of {\em Algebra and
  Applications}.
\newblock Springer, 2021.

\bibitem[Bar20]{barton}
R.~W. Barton.
\newblock {A model 2-category of enriched combinatorial premodel categories}.
\newblock \href{https://arxiv.org/abs/2004.12937}{arXiv:2004.12937v1}, 2020.

\bibitem[BB09]{intervals}
O.~Bernardi and N.~Bonichon.
\newblock Intervals in {C}atalan lattices and realizers of triangulations.
\newblock {\em J. Combin. Theory Ser. A}, 116(1):55--75, 2009.

\bibitem[BBR21]{bbr}
S.~Balchin, D.~Barnes, and C.~Roitzheim.
\newblock {$N_\infty$}-operads and associahedra.
\newblock {\em Pacific Journal of Mathematics}, 2021.

\bibitem[Bek10]{beke}
T.~Beke.
\newblock Fibrations of simplicial sets.
\newblock {\em Appl. Categ. Structures}, 18(5):505--516, 2010.

\bibitem[BH15]{blumberghill}
A.~J. Blumberg and M.~A. Hill.
\newblock Operadic multiplications in equivariant spectra, norms, and
  transfers.
\newblock {\em Adv. Math.}, 285:658--708, 2015.

\bibitem[Cha07]{chapoton}
F.~Chapoton.
\newblock Sur le nombre d'intervalles dans les treillis de {T}amari.
\newblock {\em S\'{e}m. Lothar. Combin.}, 55:Art. B55f, 18, 2005/07.

\bibitem[DS95]{dwyerspa}
W.~G. Dwyer and J.~Spali\'{n}ski.
\newblock Homotopy theories and model categories.
\newblock In {\em Handbook of algebraic topology}, pages 73--126.
  North-Holland, Amsterdam, 1995.

\bibitem[DZ21]{dz21}
J.-M. Droz and I.~Zakharevich.
\newblock Extending to a model structure is not a first-order property.
\newblock {\em New York J. Math.}, 27:319--348, 2021.

\bibitem[FOO{\etalchar{+}}21]{fooqw}
E.~E. Franchere, K.~Ormsby, A.~M. Osorno, W.~Qin, and R.~Waugh.
\newblock Self-duality of the lattice of transfer systems via weak
  factorization systems.
\newblock {\em Homology, Homotopy and Applications}, 2021.

\bibitem[Hir03]{hirschhorn}
P.~S. Hirschhorn.
\newblock {\em Model categories and their localizations}, volume~99 of {\em
  Mathematical Surveys and Monographs}.
\newblock American Mathematical Society, Providence, RI, 2003.

\bibitem[HMOO22]{hmoo}
U.~Hafeez, P.~Marcus, K.~Ormsby, and A.~M. Osorno.
\newblock Saturated and linear isometric transfer systems for cyclic groups of
  order {$p^mq^n$}.
\newblock {\em Topology Appl.}, 317:Paper No. 108162, 20, 2022.

\bibitem[Hov99]{hovey}
M.~Hovey.
\newblock {\em Model categories}, volume~63 of {\em Mathematical Surveys and
  Monographs}.
\newblock American Mathematical Society, Providence, RI, 1999.

\bibitem[JT07]{JT}
A.~Joyal and M.~Tierney.
\newblock Quasi-categories vs {S}egal spaces.
\newblock In {\em Categories in algebra, geometry and mathematical physics},
  volume 431 of {\em Contemp. Math.}, pages 277--326. Amer. Math. Soc.,
  Providence, RI, 2007.

\bibitem[Knu97]{aocp}
D.~E. Knuth.
\newblock {\em The art of computer programming. {V}ol. 1}.
\newblock Addison-Wesley, Reading, MA, 1997.
\newblock Fundamental algorithms, Third edition.

\bibitem[Lee15]{quillenstruct}
S.~Lee.
\newblock Building a model category out of cofibrations and fibrations: the two
  out of three property for weak equivalences.
\newblock {\em Theory Appl. Categ.}, 30:Paper No. 36, 1163--1181, 2015.

\bibitem[Lur21]{kerodon}
Jacob Lurie.
\newblock Kerodon.
\newblock \url{https://kerodon.net}, 2021.

\bibitem[Qui67]{quillen}
D.~G. Quillen.
\newblock {\em Homotopical algebra}.
\newblock Lecture Notes in Mathematics, No. 43. Springer-Verlag, Berlin-New
  York, 1967.

\bibitem[Rap10]{raptis}
G.~Raptis.
\newblock Homotopy theory of posets.
\newblock {\em Homology Homotopy Appl.}, 12(2):211--230, 2010.

\bibitem[Rez10]{rekz}
C.~Rezk.
\newblock A model category for categories, 2010.
\newblock Unpublished notes -
  \url{https://faculty.math.illinois.edu/~rezk/papers.html}.

\bibitem[Rub21]{rubin}
J.~Rubin.
\newblock Detecting {S}teiner and linear isometries operads.
\newblock {\em Glasg. Math. J.}, 63(2):307--342, 2021.

\bibitem[Sha76]{shapiro}
L.~W. Shapiro.
\newblock A {C}atalan triangle.
\newblock {\em Discrete Math.}, 14(1):83--90, 1976.

\bibitem[SP12]{canonical}
C.~Schommer-Pries.
\newblock The canonical model structure on \textbf{Cat}, 2012.
\newblock Blog post -
  \url{https://sbseminar.wordpress.com/2012/11/16/the-canonical-model-structure-on-cat}.

\bibitem[Sta63]{stasheff}
James~Dillon Stasheff.
\newblock Homotopy associativity of {$H$}-spaces. {I}, {II}.
\newblock {\em Trans. Amer. Math. Soc. 108 (1963), 275-292; ibid.},
  108:293--312, 1963.

\bibitem[Str72]{strom}
A.~Str{\o}m.
\newblock The homotopy category is a homotopy category.
\newblock {\em Arch. Math. (Basel)}, 23:435--441, 1972.

\end{thebibliography}

\begin{figure}
\centering
\includegraphics[scale=0.365]{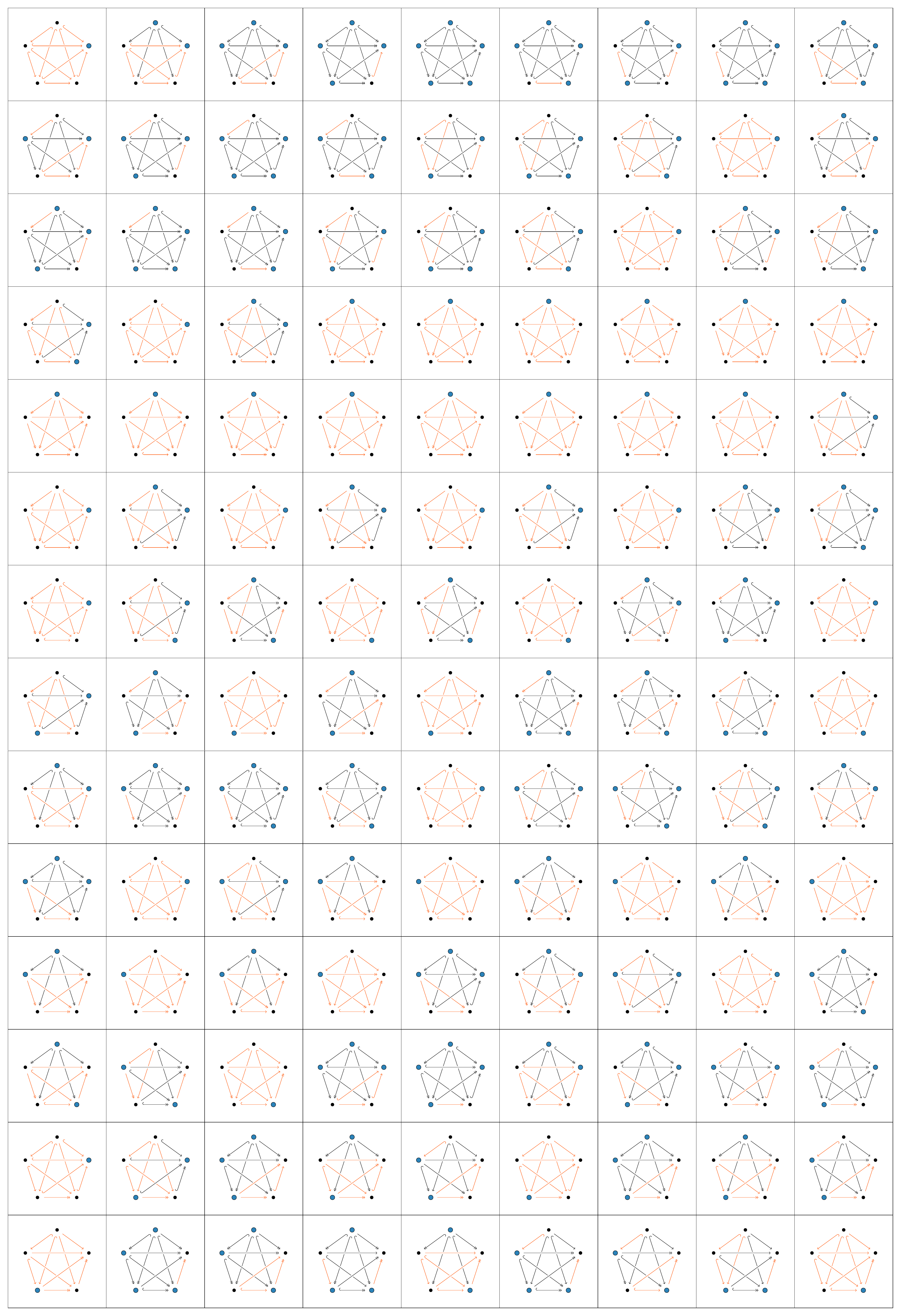}
\vspace*{-5mm}
\caption{The 126 unique Quillen model structures on the poset $[4]$. An orange arrow indicates a weak equivalence, and a larger, blue vertex indicates a bifibrant object.}\label{fig:n4}
\end{figure}

\newpage

\begin{figure}
\centering
\begin{tikzpicture}[scale=1]

\node (n1) [isosceles triangle,isosceles triangle apex angle=60,draw, rotate=90,minimum size =1.65cm] at (0,0) {};

\node (n2) [isosceles triangle,isosceles triangle apex angle=60,draw, rotate=90,minimum size =2.3cm] at (6.4,1.15) {};

\node (n3) [isosceles triangle,isosceles triangle apex angle=60,draw, rotate=90,minimum size =2.3cm] at (4.175,4.97) {};

\node (n4) [isosceles triangle,isosceles triangle apex angle=60,draw, rotate=90,minimum size =2.3cm] at (0,6.5) {};

\node (n5) [isosceles triangle,isosceles triangle apex angle=60,draw, rotate=90,minimum size =2.3cm] at (-4.175,4.97) {};

\node (n6) [isosceles triangle,isosceles triangle apex angle=60,draw, rotate=90,minimum size =2.3cm] at (-6.4,1.15) {};

\node (n7) [isosceles triangle,isosceles triangle apex angle=60,draw, rotate=90,minimum size =2.3cm] at (-5.65,-3.195) {};

\node (n8) [isosceles triangle,isosceles triangle apex angle=60,draw, rotate=90,minimum size =2.3cm] at (-2.25,-6.15) {};

\node (n9) [isosceles triangle,isosceles triangle apex angle=60,draw, rotate=90,minimum size =2.3cm] at (2.25,-6.15) {};

\node (n10) [isosceles triangle,isosceles triangle apex angle=60,draw, rotate=90,minimum size =2.3cm] at (5.65,-3.195) {};

\draw[->,thick,black,shorten <=1mm,black] (n2) -- (n1);

\draw[->, thick,black,shorten <=1mm,black] (n3) -- (n1);

\draw[->, thick,black,  shorten <=1mm,black] (n4) -- (n1);

\draw[->, thick,black,  shorten <=1mm,black] (n5) -- (n1);

\draw[->, thick,black, shorten <=1mm,black] (n6) -- (n1);

\draw[->, thick,black, shorten <=1mm,black] (n7) -- (n1);

\draw[->, thick,black, shorten <=1mm,black] (n8) -- (n1);

\draw[->, thick,black, shorten <=1mm,black] (n9) -- (n1);

\draw[->, thick,black, shorten <=1mm,black] (n10) -- (n1);

\draw[->, thick,black,shorten >=1mm, shorten <=1mm,black] (n3) -- (n2);

\draw[->, thick,black,shorten >=1mm, shorten <=1mm,black] (n3) -- (n4);

\draw[->, thick,black,shorten >=1mm, shorten <=1mm,black] (n6) -- (n2);

\draw[->, thick,black,shorten >=1mm, shorten <=1mm,black] (n6) -- (n3);

\draw[->, thick,black,shorten >=1mm, shorten <=1mm,black] (n6) -- (n4);

\draw[->, thick,black,shorten >=1mm, shorten <=1mm,black] (n6) -- (n5);

\draw[->, thick,black,shorten >=1mm, shorten <=1mm,black] (n7) -- (n2);

\draw[->, thick,black,shorten >=1mm, shorten <=1mm,black] (n7) -- (n5);

\draw[->, thick,black,shorten >=1mm, shorten <=1mm,black] (n8) -- (n2);

\draw[->, thick,black,shorten >=1mm, shorten <=1mm,black] (n8) -- (n3);

\draw[->, thick,black,shorten >=1mm, shorten <=1mm,black] (n8) -- (n4);

\draw[->, thick,black,shorten >=1mm, shorten <=1mm,black] (n8) -- (n5);

\draw[->, thick,black,shorten >=1mm, shorten <=1mm,black] (n8) -- (n6);

\draw[->, thick,black,shorten >=1mm, shorten <=1mm,black] (n8) -- (n7);

\draw[->, thick,black,shorten >=1mm, shorten <=1mm,black] (n8) -- (n9);

\draw[->, thick,black,shorten >=1mm, shorten <=1mm,black] (n8) -- (n10);

\draw[->, thick,black,shorten >=1mm, shorten <=1mm,black] (n9) -- (n2);

\draw[->, thick,black,shorten >=1mm, shorten <=1mm,black] (n9) -- (n3);

\draw[->, thick,black,shorten >=1mm, shorten <=1mm,black] (n9) -- (n4);

\draw[->, thick,black,shorten >=1mm, shorten <=1mm,black] (n9) -- (n10);

\draw[->, thick,black,shorten >=1mm, shorten <=1mm,black] (n10) -- (n4);

\begin{scope}[scale=0.75]%%model1
\node (pol) [draw=none, minimum size=1.5cm, regular polygon, regular polygon sides=3] at (0,0) {};
\foreach \n [count=\nu from 0, remember=\n as \lastn, evaluate={\nu+\lastn}] in {1,2,...,3}
\node[anchor=\n*(360/3)-180] at (pol.corner \n){};
\foreach \n in {1,2,...,3}
\draw[fill = black] (pol.corner \n) circle (2pt);
\draw[white, line width=1mm,shorten >=3mm, shorten <=3mm] (pol.corner 1) -- (pol.corner 2);
\draw[right hook->,gRed, shorten >=3mm, shorten <=3mm](pol.corner 1) -- (pol.corner 2);
\draw[white, line width=1mm,shorten >=3mm, shorten <=3mm] (pol.corner 1) -- (pol.corner 3);
\draw[right hook->,gRed, shorten >=3mm, shorten <=3mm](pol.corner 1) -- (pol.corner 3);
\draw[white, line width=1mm,shorten >=3mm, shorten <=3mm] (pol.corner 2) -- (pol.corner 3);
\draw[right hook->,gRed, shorten >=3mm, shorten <=3mm](pol.corner 2) -- (pol.corner 3);
\end{scope}

\begin{scope}[shift={(6.4,1.15)}]%%model2
\node (pol) [draw=none, minimum size=2cm, regular polygon, regular polygon sides=3] at (0,0) {};
\foreach \n [count=\nu from 0, remember=\n as \lastn, evaluate={\nu+\lastn}] in {1,2,...,3}
\node[anchor=\n*(360/3)-180] at (pol.corner \n){};
\foreach \n in {1,2,...,3}
\draw[fill = black] (pol.corner \n) circle (2pt);
\draw[white, line width=1mm,shorten >=3mm, shorten <=3mm] (pol.corner 1) -- (pol.corner 2);
\draw[right hook->>, shorten >=3mm, shorten <=3mm](pol.corner 1) -- (pol.corner 2);
\draw[white, line width=1mm,shorten >=3mm, shorten <=3mm] (pol.corner 1) -- (pol.corner 3);
\draw[right hook->>, shorten >=3mm, shorten <=3mm](pol.corner 1) -- (pol.corner 3);
\draw[white, line width=1mm,shorten >=3mm, shorten <=3mm] (pol.corner 2) -- (pol.corner 3);
\draw[right hook->,gRed, shorten >=3mm, shorten <=3mm](pol.corner 2) -- (pol.corner 3);
\end{scope}

\begin{scope}[shift={(4.175,4.97)}]%%model3
\node (pol) [draw=none, minimum size=2cm, regular polygon, regular polygon sides=3] at (0,0) {};
\foreach \n [count=\nu from 0, remember=\n as \lastn, evaluate={\nu+\lastn}] in {1,2,...,3}
\node[anchor=\n*(360/3)-180] at (pol.corner \n){};
\foreach \n in {1,2,...,3}
\draw[fill = black] (pol.corner \n) circle (2pt);
\draw[white, line width=1mm,shorten >=3mm, shorten <=3mm] (pol.corner 1) -- (pol.corner 2);
\draw[right hook->>, shorten >=3mm, shorten <=3mm](pol.corner 1) -- (pol.corner 2);
\draw[white, line width=1mm,shorten >=3mm, shorten <=3mm] (pol.corner 1) -- (pol.corner 3);
\draw[right hook->>, shorten >=3mm, shorten <=3mm](pol.corner 1) -- (pol.corner 3);
\draw[white, line width=1mm,shorten >=3mm, shorten <=3mm] (pol.corner 2) -- (pol.corner 3);
\draw[right hook->>, shorten >=3mm, shorten <=3mm](pol.corner 2) -- (pol.corner 3);
\end{scope}

\begin{scope}[shift={(0,6.5)}]%%model4
\node (pol) [draw=none, minimum size=2cm, regular polygon, regular polygon sides=3] at (0,0) {};
\foreach \n [count=\nu from 0, remember=\n as \lastn, evaluate={\nu+\lastn}] in {1,2,...,3}
\node[anchor=\n*(360/3)-180] at (pol.corner \n){};
\foreach \n in {1,2,...,3}
\draw[fill = black] (pol.corner \n) circle (2pt);
\draw[white, line width=1mm,shorten >=3mm, shorten <=3mm] (pol.corner 1) -- (pol.corner 2);
\draw[right hook->,gRed, shorten >=3mm, shorten <=3mm](pol.corner 1) -- (pol.corner 2);
\draw[white, line width=1mm,shorten >=3mm, shorten <=3mm] (pol.corner 1) -- (pol.corner 3);
\draw[right hook->, shorten >=3mm, shorten <=3mm](pol.corner 1) -- (pol.corner 3);
\draw[white, line width=1mm,shorten >=3mm, shorten <=3mm] (pol.corner 2) -- (pol.corner 3);
\draw[right hook->>, shorten >=3mm, shorten <=3mm](pol.corner 2) -- (pol.corner 3);
\end{scope}

\begin{scope}[shift={(-4.175,4.97)}]%%model5
\node (pol) [draw=none, minimum size=2cm, regular polygon, regular polygon sides=3] at (0,0) {};
\foreach \n [count=\nu from 0, remember=\n as \lastn, evaluate={\nu+\lastn}] in {1,2,...,3}
\node[anchor=\n*(360/3)-180] at (pol.corner \n){};
\foreach \n in {1,2,...,3}
\draw[fill = black] (pol.corner \n) circle (2pt);
\draw[white, line width=1mm,shorten >=3mm, shorten <=3mm] (pol.corner 1) -- (pol.corner 2);
\draw[->>,gRed, shorten >=3mm, shorten <=3mm](pol.corner 1) -- (pol.corner 2);
\draw[white, line width=1mm,shorten >=3mm, shorten <=3mm] (pol.corner 1) -- (pol.corner 3);
\draw[right hook->,gRed, shorten >=3mm, shorten <=3mm](pol.corner 1) -- (pol.corner 3);
\draw[white, line width=1mm,shorten >=3mm, shorten <=3mm] (pol.corner 2) -- (pol.corner 3);
\draw[right hook->,gRed, shorten >=3mm, shorten <=3mm](pol.corner 2) -- (pol.corner 3);
\end{scope}

\begin{scope}[shift={(-6.4,1.15)}]%%model6
\node (pol) [draw=none, minimum size=2cm, regular polygon, regular polygon sides=3] at (0,0) {};
\foreach \n [count=\nu from 0, remember=\n as \lastn, evaluate={\nu+\lastn}] in {1,2,...,3}
\node[anchor=\n*(360/3)-180] at (pol.corner \n){};
\foreach \n in {1,2,...,3}
\draw[fill = black] (pol.corner \n) circle (2pt);
\draw[white, line width=1mm,shorten >=3mm, shorten <=3mm] (pol.corner 1) -- (pol.corner 2);
\draw[->>,gRed, shorten >=3mm, shorten <=3mm](pol.corner 1) -- (pol.corner 2);
\draw[white, line width=1mm,shorten >=3mm, shorten <=3mm] (pol.corner 1) -- (pol.corner 3);
\draw[right hook->>, shorten >=3mm, shorten <=3mm](pol.corner 1) -- (pol.corner 3);
\draw[white, line width=1mm,shorten >=3mm, shorten <=3mm] (pol.corner 2) -- (pol.corner 3);
\draw[right hook->>, shorten >=3mm, shorten <=3mm](pol.corner 2) -- (pol.corner 3);
\end{scope}

\begin{scope}[shift={(-5.65,-3.195)}]%%model7
\node (pol) [draw=none, minimum size=2cm, regular polygon, regular polygon sides=3] at (0,0) {};
\foreach \n [count=\nu from 0, remember=\n as \lastn, evaluate={\nu+\lastn}] in {1,2,...,3}
\node[anchor=\n*(360/3)-180] at (pol.corner \n){};
\foreach \n in {1,2,...,3}
\draw[fill = black] (pol.corner \n) circle (2pt);
\draw[white, line width=1mm,shorten >=3mm, shorten <=3mm] (pol.corner 1) -- (pol.corner 2);
\draw[->>,gRed, shorten >=3mm, shorten <=3mm](pol.corner 1) -- (pol.corner 2);
\draw[white, line width=1mm,shorten >=3mm, shorten <=3mm] (pol.corner 1) -- (pol.corner 3);
\draw[->>,gRed, shorten >=3mm, shorten <=3mm](pol.corner 1) -- (pol.corner 3);
\draw[white, line width=1mm,shorten >=3mm, shorten <=3mm] (pol.corner 2) -- (pol.corner 3);
\draw[right hook->,gRed, shorten >=3mm, shorten <=3mm](pol.corner 2) -- (pol.corner 3);
\end{scope}

\begin{scope}[shift={(-2.25,-6.15)}]%%model8
\node (pol) [draw=none, minimum size=2cm, regular polygon, regular polygon sides=3] at (0,0) {};
\foreach \n [count=\nu from 0, remember=\n as \lastn, evaluate={\nu+\lastn}] in {1,2,...,3}
\node[anchor=\n*(360/3)-180] at (pol.corner \n){};
\foreach \n in {1,2,...,3}
\draw[fill = black] (pol.corner \n) circle (2pt);
\draw[white, line width=1mm,shorten >=3mm, shorten <=3mm] (pol.corner 1) -- (pol.corner 2);
\draw[->>,gRed, shorten >=3mm, shorten <=3mm](pol.corner 1) -- (pol.corner 2);
\draw[white, line width=1mm,shorten >=3mm, shorten <=3mm] (pol.corner 1) -- (pol.corner 3);
\draw[->>,gRed, shorten >=3mm, shorten <=3mm](pol.corner 1) -- (pol.corner 3);
\draw[white, line width=1mm,shorten >=3mm, shorten <=3mm] (pol.corner 2) -- (pol.corner 3);
\draw[->>,gRed, shorten >=3mm, shorten <=3mm](pol.corner 2) -- (pol.corner 3);
\end{scope}

\begin{scope}[shift={(2.25,-6.15)}]%%model9
\node (pol) [draw=none, minimum size=2cm, regular polygon, regular polygon sides=3] at (0,0) {};
\foreach \n [count=\nu from 0, remember=\n as \lastn, evaluate={\nu+\lastn}] in {1,2,...,3}
\node[anchor=\n*(360/3)-180] at (pol.corner \n){};
\foreach \n in {1,2,...,3}
\draw[fill = black] (pol.corner \n) circle (2pt);
\draw[white, line width=1mm,shorten >=3mm, shorten <=3mm] (pol.corner 1) -- (pol.corner 2);
\draw[right hook->>, shorten >=3mm, shorten <=3mm](pol.corner 1) -- (pol.corner 2);
\draw[white, line width=1mm,shorten >=3mm, shorten <=3mm] (pol.corner 1) -- (pol.corner 3);
\draw[->>, shorten >=3mm, shorten <=3mm](pol.corner 1) -- (pol.corner 3);
\draw[white, line width=1mm,shorten >=3mm, shorten <=3mm] (pol.corner 2) -- (pol.corner 3);
\draw[->>,gRed, shorten >=3mm, shorten <=3mm](pol.corner 2) -- (pol.corner 3);
\end{scope}

\begin{scope}[shift={(5.65,-3.195)}]%%model10
\node (pol) [draw=none, minimum size=2cm, regular polygon, regular polygon sides=3] at (0,0) {};
\foreach \n [count=\nu from 0, remember=\n as \lastn, evaluate={\nu+\lastn}] in {1,2,...,3}
\node[anchor=\n*(360/3)-180] at (pol.corner \n){};
\foreach \n in {1,2,...,3}
\draw[fill = black] (pol.corner \n) circle (2pt);
\draw[white, line width=1mm,shorten >=3mm, shorten <=3mm] (pol.corner 1) -- (pol.corner 2);
\draw[right hook->,gRed, shorten >=3mm, shorten <=3mm](pol.corner 1) -- (pol.corner 2);
\draw[white, line width=1mm,shorten >=3mm, shorten <=3mm] (pol.corner 1) -- (pol.corner 3);
\draw[->,gRed, shorten >=3mm, shorten <=3mm](pol.corner 1) -- (pol.corner 3);
\draw[white, line width=1mm,shorten >=3mm, shorten <=3mm] (pol.corner 2) -- (pol.corner 3);
\draw[->>,gRed, shorten >=3mm, shorten <=3mm](pol.corner 2) -- (pol.corner 3);
\end{scope}
\end{tikzpicture}
\vspace*{10mm}
\caption{The collection of model structures on $[2]$, where an arrow indicates that the identity functor is left Quillen. We have used an orange arrow to indicate a weak equivalence.}\label{2quillen}
\end{figure}
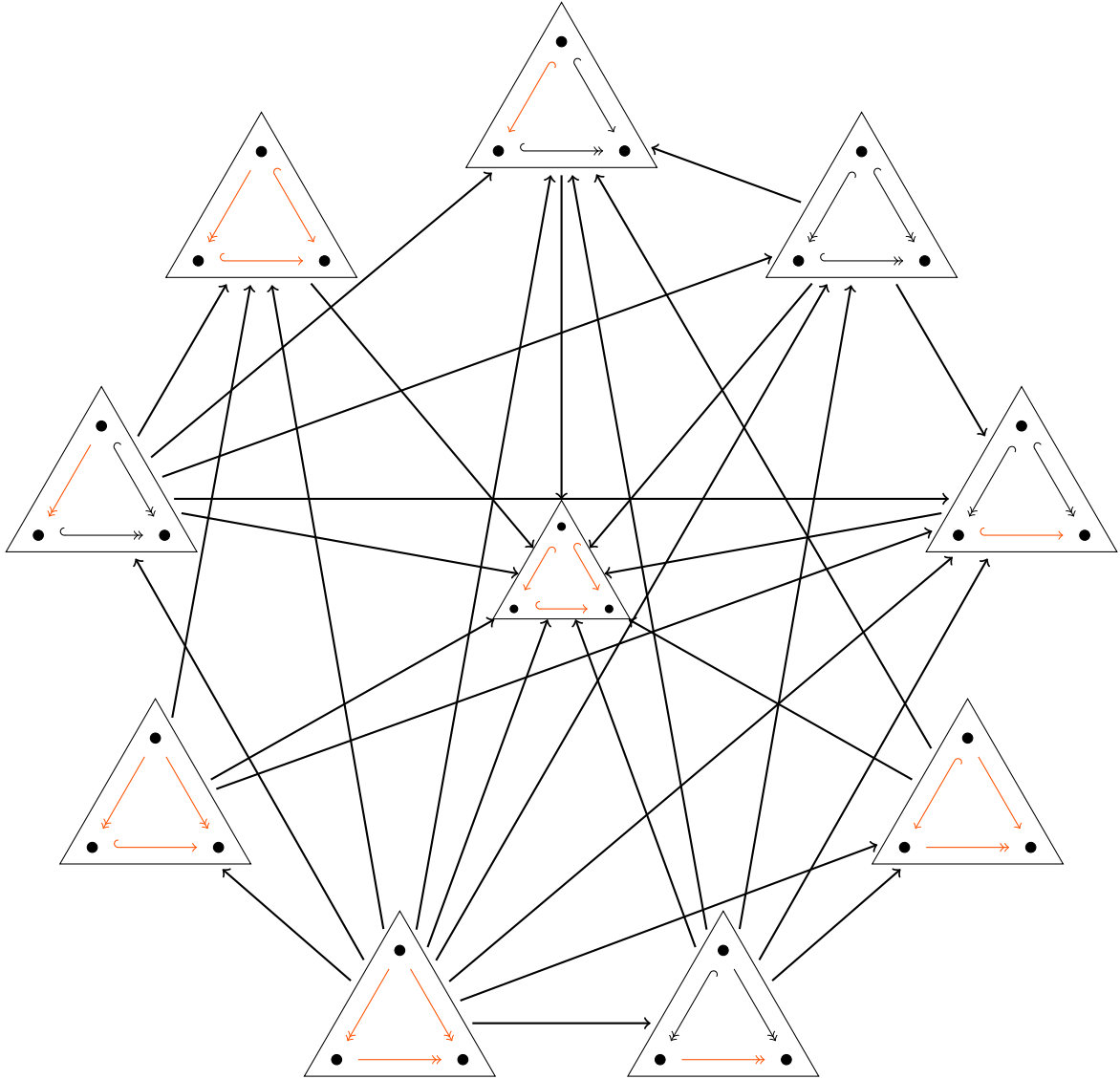

\end{document}